\documentclass[review,hidelinks,onefignum,onetabnum,reqno]{siamart220329}

\usepackage{amsfonts}
\usepackage{graphicx}
\usepackage{epstopdf}
\usepackage{algorithmic}

\usepackage{amsmath,amssymb}
\usepackage{enumitem}
\usepackage{tikz-cd}
\usepackage{bbold}
\usepackage{xcolor}
\usepackage{mathtools}

\ifpdf
  \DeclareGraphicsExtensions{.eps,.pdf,.png,.jpg}
\else
  \DeclareGraphicsExtensions{.eps}
\fi

\newsiamremark{remark}{Remark}
\newsiamremark{hypothesis}{Hypothesis}
\newsiamthm{claim}{Claim}

\headers{Learning Time-Stepping Algorithms for PDEs}{M. Krishnan, K. Chen, and H. Yang}

\title{error analysis for learning the time-stepping operator of evolutionary PDEs\thanks{Submitted to the editors DATE.
\funding{KC and HY were partially supported by the US National Science Foundation under awards DMS-2244988, DMS-2206333, the Office of Naval Research Award N00014-23-1-2007, and the DARPA D24AP00325-00.}}}

\author{Ke Chen\thanks{Department of Mathematical Sciences, University of Delaware, DE 
  (\email{kechen@udel.edu}).}
\and Meenakshi Krishnan\thanks{Department of Mathematics, University of Maryland College Park, MD 
  (\email{mkrishn9@umd.edu}).}
\and Haizhao Yang\thanks{Department of Mathematics and Department of Computer Science, University of Maryland College Park, MD 
  (\email{hzyang@umd.edu}).}}

\usepackage{amsopn}

\newsiamremark{assumption}{Assumption}

\newcommand{\Xcal}{\mathcal{X}}
\newcommand{\Ycal}{\mathcal{Y}}

\newcommand{\FNN}{\mathcal{F}_{\mathrm{NN}}}
\newcommand{\Egen}{\mathcal{E}_{gen}(\GNN)}
\newcommand{\EgenN}{\mathcal{E}^N_{gen}(\GNN)}
\newcommand{\Ep}{\mathcal{E}_\text{proj}}

\newcommand{\GNN}{\Gamma_{\mathrm{NN}}}

\usepackage{setspace}
\usepackage{listings}
\usepackage{color}

\usepackage{tikz}
\usetikzlibrary{arrows.meta,positioning,decorations.pathreplacing}

\usepackage{centernot}
\usepackage{mathtools}
\usepackage{stmaryrd}

\nolinenumbers 

\begin{document}
 
\maketitle

\begin{abstract}
Deep neural networks (DNNs) have recently emerged as effective tools for approximating solution operators of partial differential equations (PDEs) including evolutionary problems. Classical numerical solvers for such PDEs often face challenges of balancing stability constraints and the high computational cost of iterative solvers. In contrast, DNNs offer a data-driven alternative through direct learning of time-stepping operators to achieve this balancing goal. In this work, we provide a rigorous theoretical framework for analyzing the approximation of these operators using feedforward neural networks (FNNs). We derive explicit error estimates that characterize the dependence of the approximation error on the network architecture - namely its width and depth - as well as the number of training samples. Furthermore, we establish Lipschitz continuity properties of time-stepping operators associated with classical numerical schemes and identify low-complexity structures inherent in these operators for several classes of PDEs, including reaction–diffusion equations, parabolic equations with external forcing, and scalar conservation laws. Leveraging these structural insights, we obtain generalization bounds that demonstrate efficient learnability without incurring the curse of dimensionality. Finally, we extend our analysis from single-input operator learning to a general multi-input setting, thereby broadening the applicability of our results.

\end{abstract}




\section{Introduction}
Partial Differential Equations (PDEs) are fundamental models for many applications in physics, biology, and engineering.
Of particular importance are evolutionary PDEs, which arise in applications like fluid dynamics \cite{temam2024navier} \cite{monk2003finite}, and wave propagation \cite{durran2013numerical}. 
Classical time-stepping methods for these problems include explicit and implicit schemes \cite{leveque1992numerical, butcher1964implicit}. 
Explicit methods, while straightforward to implement, suffer from severe stability constraints, necessitating prohibitively small time-steps. Conversely, implicit methods offer improved stability but demand solving complex nonlinear systems at each step, often using iterative solvers. While Implicit-Explicit (IMEX) schemes are also popular, achieving both high accuracy and stability with them can be challenging~\cite{hu2021uniform}. Consequently, these approaches can be computationally inefficient for nonlinear PDEs, making long-time simulations infeasible.

In recent years, Deep Neural Networks (DNNs) have emerged as a powerful alternative to classical numerical schemes for solving PDEs~\cite{lagaris1998artificial,han2018solving, raissi2019physics, sirignano2018dgm}, due to their superior ability to approximate complex functions \cite{chen1993approximations,yarotsky2018optimal,SHEN2022101,JMLR:v23:21-1404,SHEN2021160} and operators between function spaces \cite{liu2024deep,chen2023deep,li2023neural,lu2021learning, lu2019deeponet,kovachki2024datacomplexityestimatesoperator,mhaskar2022localapproximationoperators}. 
A recent trend is utilizing neural networks to learn PDE operators, which map initial conditions, source terms into solution fields in the forward setting, and conversely, mapping observed data to model parameters in the inverse setting. Many successful models have been proposed, including PCA-net~\cite{bhattacharya2021model,lanthaler2023operator}, Graph neural operators (GNN)~\cite{li2020neural,scarselli2008graph,battaglia2018relational}, DeepONet~\cite{lu2019deeponet}, Fourier neural operators (FNO)~\cite{li2020fourier,kovachki2023neural}, Physics-informed neural operators (PINO)~\cite{li2024physics}, and many others~\cite{wang2021learning,geng2024deep,mattey2022novel,yang2024pde,ong2022iae}. 

We are interested in operator learning frameworks for time-dependent PDEs. Consider a general non-linear PDE of the form:
\begin{align} \label{eqn:gen_pde} 
\begin{cases}
    \partial_t u(x,t) + \mathcal{L} (x, u,\mathcal{D}u, \dots,\mathcal{D}^k u) = f(u,x), \quad \text{ for } x\in D \subseteq \mathbb{R}^d,\ t \in (0,T], \\
    u(x,0) = u^0(x),\ x \in D,
    \end{cases}
\end{align}
with suitable boundary conditions imposed. Here, $\mathcal{D}$ represents the partial derivative operator, and $\mathcal{L}: D \times \mathbb{R} \times \mathbb{R}^d \times \dots \times \mathbb{R}^{d^k} \to \mathbb{R}$ with $k>0$ is a general nonlinear function. The function $f : \mathbb{R} \times D  \to \mathbb{R}$ represents the forcing/reaction term. A numerical method for \eqref{eqn:gen_pde} advances the solution from $u^n$ at time $t^n$ to $u^{n+1}$ at $t^{n+1} = t^n + \Delta t$. A general $s$-stage Runge-Kutta (RK) method defines this evolution as:
\begin{align}
    u^{n+1}(x) = u^{n}(x) + \Delta t \sum_{i=1}^{s} {b}_i k_i(x) &,\  \label{eqn:rk_gen0}\\
   k_i(x) = -\mathcal{L}\left(x, U^i, \mathcal{D} U^i, \dots \right) + f\left(U^i, x\right), \text{ for } U^i &= u^{n}(x) + \Delta t \sum_{j=1}^{s} a_{ij} k_j(x).
    \label{eqn:rk_gen}
\end{align}
Here, $ a_{ij},\ {b}_i$ for $1\leq i,j\leq s$ are coefficients of the RK method. This defines a time-stepping operator, $\Phi_{\Delta t}: u^n\mapsto u^{n+1}$. For implicit methods (the matrix $\left[a_{ij}\right]$ is not strictly lower triangular), computing $u^{n+1}$ requires expensive iterative solvers.

Several deep-learning based methods have been proposed to obtain a solution to \eqref{eqn:gen_pde}. The solution usually does not admit an explicit representation, and methods that directly learn the solution using a neural network, may suffer from low accuracy due to global-in-time optimization \cite{chenteng}. While sequential-in-time training with time-dependent neural network weights were proposed~\cite{chenteng,berman2024randomized,chen2023implicit}, obtaining the terminal solution of \eqref{eqn:gen_pde} still requires efficiently adjusting the neural network weights for all time steps, leading to intractable computational cost. In this paper, we consider an alternative: learning the semi-discrete time-stepping operator $\Phi_{\Delta t}: u^n \mapsto u^{n+1}$ that marches the PDE solution to the next time-step with a \textit{time-independent} network.  
A neural network proxy for this operator can bypass the costly nonlinear solves of implicit methods, which is particularly appealing for long-time simulations where the cumulative cost of iterative solvers becomes a bottleneck.  
This idea was numerically investigated in~\cite{qin2019data} for Ordinary Differential Equations (ODEs) with ResNet, and extended to stochastic differential equations (SDEs)~\cite{chen2024learning}.
It also has been explored for various PDEs ~\cite{raissi2019physics, zhai2023parameter, zhu2023convolutional,geng2024deep}, where a learned time-stepping map predicts the solution trajectory on a fixed spatial mesh.

A natural extension of this framework is to learn operators that take in multiple inputs, such as the operator $\Psi_{\Delta t}:(u^n,f) \mapsto u^{n+1}$  mapping both the initial state and the reaction/forcing term to the numerical solution after a fixed time-step. This leads to the concept of Multi-Input Operator learning \cite{jin2022mionet,jiang2024fourier} which considers operator regression via neural networks for operators defined on a product of Banach spaces.

A major hurdle in learning functions or operators with neural networks is the curse of dimensionality (CoD)~\cite{bauer2019deep,chen2019efficient, hutzenthaler2020proof,nakada2020adaptive,poggio2017and,bach2017breaking}, when increasing the degree of freedom $n$ correspondingly increases the computational complexity in an exponential manner, making them intractable for large-scale problems. Recent research \cite{wojtowytsch2020can,yang2022approximation,chen2021representation,darbon2020overcoming,chen2023deep} has shown that DNNs can mitigate CoD when learning operators with certain low-complexity structures. However, these theories often assume a general Lipschitz operator with certain implicit low dimensional structures (e.g. in a Barron space~\cite{barron1991universal,wojtowytsch2020can}), and are thus not readily applicable to time-stepping operators for PDEs.

\textbf{Main Contributions:} We establish generalization error bounds for learning time-stepping operators of PDEs using deep neural networks (DNNs), with the goal of understanding how operator learning can bypass the computational costs of traditional solvers while avoiding CoD. We focus on three important PDE classes: reaction-diffusion equations, parabolic equations with forcing terms, and viscous conservation laws. Our analysis reveals that the corresponding time-stepping operators possess a formally defined \emph{low-complexity structure}, namely, they can be decomposed into a sequence of linear transformations and low-dimensional nonlinear operations. 

To support this, we derive Lipschitz estimates using tools from classical numerical analysis, including semigroup theory and convergence guarantees of numerical schemes. These estimates, together with the identified structure, allow us to prove that the generalization error scales polynomially with the encoding dimension of the input, thereby avoiding the exponential complexity typical of CoD.

For reaction-diffusion equations, we analyze implicit Euler methods with Picard and Newton solvers. We also study implicit Euler and Crank-Nicolson methods for parabolic equations with forcing, and implicit Euler with Picard solver for viscous conservation laws. Across all cases, we derive explicit error bounds, and show that DNNs can efficiently learn time evolution with theoretical performance guarantees.

A further contribution is our analysis of multi-input operator networks acting on product spaces. Unlike previous works focusing on approximation \cite{jin2022mionet,jiang2024fourier}, our study establishes generalization error bounds. The low-complexity structure of multi-input time-stepping operators provides insight into how neural networks can effectively handle operator regression tasks involving interacting input functions.

\textbf{Organization:} Section \ref{sec:prob_form} formulates the operator learning problem for the time-stepping operator. Sections \ref{sec:rdeq}, \ref{sec:para_gen} and \ref{sec:cons_gen} present our analysis of the generalization error for learning time-stepping operators for reaction-diffusion equations, parabolic systems with forcing, and viscous conservation laws, respectively. We discuss global error accumulation in Section \ref{sec:global_error} and conclude with future work in Section \ref{sec:concl}.

\section{Problem Formulation}\label{sec:prob_form}
This section introduces the paper's notation and mathematical framework. The norm of a Banach space $\Xcal$ is denoted by $\|\cdot\|_{\Xcal}$. A product space ${{\Xcal_1} \times {\Xcal_2}}$, is equipped with the sum norm: $\|(u,f)\|_{{\Xcal_1} \times {\Xcal_2}} = \|u\|_{{\Xcal_1}} + \|f\|_{{\Xcal_2}} $ for $u \in \Xcal_1$, $f \in \Xcal_2$. We use standard asymptotic notation: $x=\Omega(y)$ signifies $x\geq Cy$, for some constant $C>0$, and $x=\mathcal{O}(y)$ is the corresponding upper bound.

\subsection{Problem} 
Our objective is to learn the PDE time-stepping operator. We consider two cases:
\begin{itemize}
    \item Single-input: The target operator $\Phi_{\Delta t}:\Xcal_1\to \Ycal$ maps an initial state $u^0 \in \Xcal_1=\Xcal$ to the state $u^1\in \Ycal$ after fixed time-step $\Delta t> 0$. 
    \item Multi-input: The target operator $\Psi_{\Delta t}:\Xcal_1\times \Xcal_2 \to  \Ycal$ maps a tuple $(u^0,f) \in \Xcal_1\times \Xcal_2= \Xcal$, comprising the initial state, and a model parameter (e.g. the forcing function) to the state $u^1\in \Ycal$ after time-step $\Delta t$.
\end{itemize}
For our problem formulation, we will use the unified notation $\mathcal{S}_{\Delta t}: {\Xcal}  \to {\Ycal}$, to refer to either operator, where the domain $\Xcal$ is understood from the context.

In both cases, the target non-linear operator $\mathcal{S}_{\Delta t}$ is approximated using a DNN trained on a dataset $ \{ ({u}_i, v_i)  \mid  v_i = \mathcal{S}_{\Delta t}(u_i) ,\ i=1,\ldots,n \}$ generated independently and identically distributed (i.i.d.) from a random measure $\gamma$ over ${\Xcal}$. We assume that the data is noise-free, since it is generated using deterministic numerical algorithms. 

Since DNNs operate in finite-dimensional spaces,
we use encoder-decoder pairs, specifically $E_{\mathcal{P} }: \mathcal{P}  \to \mathbb{R}^{d_{\mathcal{P}  }}$ and $D_{\mathcal{P}   }: \mathbb{R}^{d_{\mathcal{P} }}\to {\mathcal{P}  }$, where $d_{\mathcal{P} }$ is the encoding dimension for $\mathcal{P}=\Xcal, \Ycal$. For product space $\Xcal = {{\Xcal_1} \times {\Xcal_2}}$, the encoder is defined by concatenating the individual encodings, $E_{{\Xcal_1} \times {\Xcal_2}}(u^0,f) = \begin{bmatrix} E_{{\Xcal_1}}(u^0); E_{ {\Xcal_2}}(f) \end{bmatrix}$ where $E_{\Xcal_i}: {\Xcal_i}\to\mathbb{R}^{d_{\Xcal_i}}$ for $i=1,2$, with $d_{\Xcal} = d_{\Xcal_1} + d_{\Xcal_2}$. The decoder is defined analogously.

 The time-stepping map is thus approximated using a finite-dimensional operator $\Gamma: \mathbb{R}^{d_\Xcal} \to \mathbb{R}^{d_\Ycal}$ composed with encoders and decoders so that $\mathcal{S}_{\Delta t} \approx D_\Ycal  \circ \Gamma \circ E_{{\Xcal}}$.
 This can be achieved by solving the following optimization problem:%
\begin{equation}\label{eqn:optimization}
	\underset{\Gamma\in 
	\FNN}{\mathrm{argmin}} \frac{1}{n} \sum_{i=1}^{n} \| \Gamma\circ E_{{\Xcal}}(u_i) -E_\Ycal(v_i)\|_2^2.
\end{equation}
The class $\FNN$ consists of rectified linear unit (ReLU) feedforward DNNs of the form,%
\begin{equation}\label{eqn:fnn} 
     f(x)  =  W_L\phi_{L-1} \circ \phi_{L-2} \circ\cdots \circ \phi_1(x) + \beta_L\,,\ \  \phi_i(x) \coloneqq \sigma(W_i x + \beta_i) \,, i =1,\ldots,L-1 \,,
\end{equation}
where $\sigma(x) = \max\{x,0\}$ is the ReLU activation function evaluated point-wise, and $W_i$ and $\beta_i$ represent weight matrices and bias vectors, respectively. 
In practice, the space $\FNN$ is selected as a compact set comprising ReLU feedforward DNNs. The following architecture is considered within the class of $\FNN$ functions:
\begin{equation}\label{eqn:FNN}
	\begin{aligned}
		&\FNN(d,L,p,M)=\{\Gamma=[f_1, f_2, \dots,f_{d}]^{\top}: 
		\mbox{ for each }k=1,\dots,d\,,f_k(x) \mbox{ is} \\
  & \mbox{in the form of (\ref{eqn:fnn}) } 
		\mbox{with  }L \mbox{ layers, width bounded by } p, 
		\|f_k\|_{\infty}\leq M\},
	\end{aligned}
\end{equation}
where $\|f\|_\infty = \sup_{x} |f(x)|$. When there is no ambiguity, the notation $\FNN$ is used without its associated parameters. 

\subsection{Assumptions} As in \cite{liu2024deep,chen2023deep}, we rely on the following assumptions to analyze the approximation and generalization error.
\begin{assumption}[Compactly supported measure]\label{assump:compact_supp}
	The probability measure $\gamma$ is supported on a compact set $\Omega_\Xcal \subset \Xcal$. Consequently, there exists $R_{{\Xcal}}>0$ such that $\|u\|_{{\Xcal}} \leq R_{{\Xcal} }$ for any $u\in\Omega_{{\Xcal} }$. Initial conditions are drawn from a compact set ${\Omega_0 \subset \Xcal_1}$ and model parameters, if any, from compact set ${\Omega_p \subset \Xcal_2}$. The input support is thus, ${\Omega_\Xcal = \Omega_0}$ for the single-input case, and ${\Omega_\Xcal = \Omega_0 \times \Omega_p}$ for the multi-input case. 
 \end{assumption}

 We denote the pushforward measure of $\gamma$ under $\mathcal{S}_{\Delta t}$ as $\mathcal{S}_{{\Delta t}{\#}\gamma}$ such that for any $\Omega \subset \Ycal$, we have $\mathcal{S}_{{\Delta t}{\#}\gamma}(\Omega)=\gamma(\{u:\mathcal{S}_{\Delta t}(u)\in \Omega\})$.

\begin{assumption}[Fixed Lipschitz encoders and decoders]\label{assump:lip_enco} The encoder-decoder pairs $E_{{\Xcal} },\ D_{{\Xcal} }$ and $E_{\Ycal},\ D_{\Ycal}$ are pre-trained and fixed. They satisfy:
$E_\mathcal{P}(0_\mathcal{P}) = \mathbf{0}\,, D_\mathcal{P}(\mathbf{0}) = 0_\mathcal{P}\,$
where $\mathbf{0}$ and $0_{\mathcal{P} }$ denote the zero vector in $\mathbb{R}^{d_\mathcal{P}}$, and the zero element in $\mathcal{P}$, respectively, for $ \mathcal{P}= \Xcal, \Ycal$. Furthermore, the encoders are Lipschitz operators:
 \[
 \| E_\mathcal{P} (u_1) - E_\mathcal{P} (u_2)\|_2 \leq L_{E_\mathcal{P}} \| u_1-u_2 \|_\mathcal{P} \,, \quad \mathcal{P} = {\Xcal}, \Ycal \,,  
	\]
 where $\|\cdot\|_2$ denotes the Euclidean $L^2$ norm, $\|\cdot\|_\mathcal{P}$ denotes the associated norm of $\mathcal{P}$. Similarly, the decoders $D_\mathcal{P},\ \mathcal{P} = {{\Xcal} },\Ycal$ are $L_{D_\mathcal{P}}$-Lipschitz continuous.
 \end{assumption}

We define $\Pi_{{{\Xcal} },d_{{\Xcal} }}:= D_{{{\Xcal} }} \circ E_{{{\Xcal} }}$ and $\Pi_{\Ycal,d_\Ycal} := D_\Ycal \circ E_\Ycal$ to denote the encoder-decoder projections on $\Xcal$ and ${{\Ycal}}$ respectively.
\begin{assumption}[Lipschitz operator]\label{assump:Lipschitz}
	The target operator $\mathcal{S}_{\Delta t}$ is Lipschitz continuous with constant $L_{\mathcal{S}_{\Delta t}}$. 
\end{assumption}

Assumptions \ref{assump:compact_supp}, \ref{assump:Lipschitz} imply that the image-set $\Omega_1:=\mathcal{S}_{\Delta t}(\Omega_\Xcal)$ is bounded by $R_\Ycal := L_{\mathcal{S}_{\Delta t}} R_{{\Xcal} }$.

Many PDE operators also possess an inherent structure allowing them to be decomposed into linear operations followed by low-dimensional non-linearities. We formalize this as a \emph{low complexity structure}.

\begin{definition}[Low Complexity Structure]
An operator $\Phi: \Xcal \to \Ycal$ has a low complexity structure with respect to a compact set $\Omega_\Xcal \subset \Xcal$ and chosen encoder-decoders ($E_{\Xcal}, D_{\Xcal}, E_{\Ycal}, D_{\Ycal}$) if, for any $u \in \Omega_{\Xcal}$, there exist $0 < d_0, \dots, d_k \leq d_{\Xcal}$ and $0 < l_0, \dots, l_k \lesssim \max\{d_{\Xcal}, d_\Ycal\}$ such that:
\[
\Pi_{\Ycal,d_\Ycal} \circ \Phi(u) = D_\Ycal \circ G^k \circ \cdots \circ G^1 \circ E_{\Xcal}(u).
\]
Here $G^i: \mathbb{R}^{\ell_{i-1}}\to \mathbb{R}^{\ell_i}$ is defined as $G^i(a) = \begin{bmatrix}
        g_1^i( (V_1^i)^\top a),\cdots, g_{\ell_i}^i((V_{\ell_i}^i)^\top a)
    \end{bmatrix}\,,$
for the matrix $V_j^i\in \mathbb{R}^{\ell_{i-1} \times d_{i}}$ and non-linearity $g_j^i: \mathbb{R}^{d_i} \to 
\mathbb{R}$ for $j=1,\ldots,\ell_i$, $i=1,\ldots,k$.
    \label{defn:low_comp}
\end{definition}

\begin{assumption}[Low Complexity Structure]\label{assump:low_complexity} The target operator $\mathcal{S}_{\Delta t}$ for the chosen encoder-decoders possesses the low complexity structure in Definition \ref{defn:low_comp}.
\end{assumption}

\subsection{Generalization Error} Given the trained neural network $\GNN$, we denote its generalization error as,
\begin{equation}
    \label{eqn:gen_err}
    \Egen := \mathbb{E}_{u\sim \gamma} \left[ \| D_{\Ycal} \circ \GNN \circ E_{\Xcal}(u) - \mathcal{S}_{\Delta t}(u)\|_{\Ycal}^2 \right]  \,.
\end{equation}
The following theorem from \cite{chen2023deep} bounds this error under the stated assumptions.

\begin{theorem}\label{thm:low_complexity}
Suppose Assumption \ref{assump:compact_supp}-\ref{assump:low_complexity} hold. 
Let $\GNN$ be the minimizer of the optimization (\ref{eqn:optimization}) for $\FNN(d_\Ycal,kL,p,M)$ in (\ref{eqn:FNN}) with parameters: 
	\begin{equation*}
		Lp =  \Omega \left( d_{\Xcal}^{\frac{4-d_{\text{max}}}{4+2d_{\text{max}}}} n^{\frac{d_{\text{max}}}{4+2d_{\text{max}}}} \right)
  \,, M \geq \sqrt{\ell_{\text{max}}} L_{E_\Ycal} R_\Ycal.
	\end{equation*}
	Here, $d_\text{max} = \max\{d_i\}_{i=1}^k$ and $\ell_\text{max} = \max\{\ell_i\}_{i=1}^k$ are parameters from Definition \ref{defn:low_comp}.
	Then we have:
	\begin{equation*}
		\begin{aligned}
			\mathcal{E}_{gen}(\GNN)  \lesssim L_{\mathcal{S}_{\Delta t}}^2 \log(L_{\mathcal{S}_{\Delta t}})\ell_\text{max}^{\frac{8+d_\text{max}}{2+d_\text{max}} } n^{-\frac{2}{2+d_\text{max}}} \log n + \mathcal{E}_{proj} \,,
		\end{aligned}
	\end{equation*}
	where the constants in $\lesssim$ and $\Omega(\cdot)$ only depend on $k,p,\ell_\text{max},R_\Ycal,L_{E_\Ycal},L_{D_\Ycal},L_{E_{{\Xcal}}}, L_{D_{{\Xcal}}}$. Here $\Ep$ represents a projection error term along with a negligible term $n^{-1}$:
 \[
\Ep \coloneqq L^2_{\mathcal{S}_{\Delta t}}   \mathbb{E}_{u\sim \gamma}  \left[  \| \Pi_{\Xcal,d_\Xcal}(u)-u \|_\Xcal^2\right]  + \mathbb{E}_{w\sim \mathcal{S}_{{\Delta t}{\#}}\gamma} \left[ \| \Pi_{\Ycal,d_\Ycal}(w) -w \|_\Ycal^2\right] 
			+ n^{-1} \, .
\]   
\end{theorem}

Assumptions \ref{assump:compact_supp}-\ref{assump:low_complexity}, together with Theorem \ref{thm:low_complexity}, provide the foundation for our analysis. Subsequent sections will show that the assumptions are met for time-stepping schemes applied to various PDE classes. This requires deriving explicit Lipschitz estimates and identifying low-complexity structure. The encoder-decoder pair for the initial condition plays a significant role in proving the latter, and we briefly discuss suitable encodings for the input space.

\subsection{Encoding the Initial Condition} 
A straightforward choice of encoder satisfying Assumption \ref{assump:lip_enco} is the discretization encoder, where the function is represented by its values on a grid in the domain. Common alternatives include basis encoders, such as those using the Fourier series
with trigonometric basis, or PCA with data-driven basis. For input functions in Hölder space $C^s(D)$ for $s\in \mathbb{R}^+$ with a bounded H\"older norm, the standard spectral encoder may be used, and the encoding/decoding projection errors may be derived from the results presented in \cite{schultz1969l}. We refer to \cite{liu2024deep} for a detailed treatment on encodings for function spaces.

With this framework established, we proceed to analyze the generalization error for learning time-stepping operators for various classes of PDEs.

\section{Reaction-Diffusion Equations} \label{sec:rdeq}
Reaction-diffusion PDEs describe phenomena such as population dynamics \cite{berestycki2008reaction} and the evolution of chemical concentrations. Examples include the Fisher-KPP equation, modeling biological species growth \cite{kolmogorov1991study}, and the Allen-Cahn equation \cite{allen1972ground}, representing phase separation in alloys.
Consider a general reaction-diffusion equation of the form:
\begin{align}
\begin{cases}
    \partial_t u(x,t)  = \Delta u(x,t) + f(u(x,t)),\ & x\in D \subseteq \mathbb{R}^d,\ t \in [0,T], \label{eqn:gen_rd} \\
    u(x,0) = u^0(x),\ &x \in D,
\end{cases}
\end{align}
where $u(x,t)$ is the state variable, the reaction function $f \in C^{1}(\mathbb{R})$. We assume homogeneous Dirichlet boundary conditions on a sufficiently smooth boundary $\partial D$. 

Consider the semi-discrete form of equation (\ref{eqn:gen_rd}) obtained by discretizing only in time. The update step for the explicit Forward Euler scheme is given by,
\begin{align*}
    u^{1} = u^0 + \Delta t (\Delta u^0 + f(u^0)) =  [1+\Delta t \Delta ]u^0  + \Delta t f(u^0).
\end{align*}
Similarly, for the implicit Backward Euler (BE) scheme, we have:
\begin{align} 
    u^{1} &= u^0 + \Delta t (\Delta u^{1} + f(u^{1})), \nonumber \\
    \implies u^{1} &= \left[1-\Delta t \Delta \right]^{-1} \left( u^0 + \Delta t f(u^{1})\right) =: \Phi_{\text{BE}}(u^{1},u^0,f).\label{eqn:implicit}
\end{align}
Let $(C_0(\bar{D}),\|\cdot\|_\infty)$ denote the Banach space of continuous functions on $\bar{D}$ vanishing on $\partial D$. Then, $\Phi_{\text{BE}}: C_0(\bar{D})\times C_0(\bar{D}) \times C^1(\mathbb{R}) \to C_0(\bar{D})$ represents the BE operator. 

The solution $u^{1}$ is the fixed point of the mapping $z \to \Phi_{\text{BE}}(z,u^0,f)$. This mapping is a contraction when $\Delta t$ is sufficiently small relative to the Lipschitz constant of $f$ (see Lemma \ref{lem:pic_lc}). The Banach fixed-point theorem thus guarantees the existence of a unique solution, which can be found using Picard iterations:
\begin{align}
    u^{(0)} = u^0, \quad 
    u^{(i)} &=  \left[1-\Delta t \Delta \right]^{-1} \left( u^0 + \Delta t f(u^{(i-1)})\right),\ 1\leq i\leq m,\quad   u^{1} \approx u^{(m)}.
\end{align}
The iterations are terminated after $m$ steps (or after reaching an error tolerance). Note that the first iteration corresponds to a first-order IMEX scheme \cite{ascher1997implicit} where linear terms are treated implicitly, and non-linear terms, explicitly.

In the single-input case with a fixed $f \in C^1(\mathbb{R})$, the target time-stepping operator $\Phi_{\text{P}}:\Omega_0 \subset C_0(\bar{D}) \to C_0(\bar{D})$ for compact $\Omega_0$ may be defined as: 
\begin{equation} \label{eqn:sing_pic_iter}
    u^{(m)} =: \Phi_{\text{P}}(u^0)
\end{equation} 
In the multi-input case, $\Psi_{\text{P}}: \Omega_0 \times \Omega_p \subset C_0(\bar{D})\times C^1(\mathbb{R}) \to C_0(\bar{D})$ represents the target BE operator with Picard iterations, where it is defined as:
\begin{equation} \label{eqn:pic_iter}
     u^{(m)} =: \Psi_{\text{P}}(u^0,f).
\end{equation}
Alternatively, Newton's method solves \eqref{eqn:implicit} by computing the root of the map:
\begin{align}
    \Phi_{u^0}(u) &:= u- [1-\Delta t \Delta ]^{-1} u^0 - [1-\Delta t \Delta ]^{-1} \Delta t f(u). \label{eqn:root_newt}
\end{align}
The solution is recovered using the iterative method given by:
\begin{align}
    u^{(0)} = u^0, \quad u^{(i)} &= u^{(i-1)} - \left[D\Phi_{u^0}\left(u^{(i-1)}\right)\right]^{-1} \Phi_{u^0}(u^{(i-1)}), \ \text{for } i = 1,\ldots,m,\nonumber\\
    u^{1} &\approx u^{(m)} =: \Phi_{\text{N}}(u^0). \label{eqn:newton}
\end{align}
Here, $\Phi_{\text{N}}: \Omega_0 \subset C_0(\bar{D})\to C_0(\bar{D})$ denotes the single-input BE operator with Newton's solver, for compact $\Omega_0$ as per Assumption \ref{assump:compact_supp}. The multi-input operator may be defined correspondingly. Newton's method converges quadratically  near the solution, under certain smoothness assumptions on $\Phi_{u^0}$ (as formalized by the Kantarovich theorem) in a neighborhood of the iterates and the solution. 
 
Our goal is to bound the generalization error for networks that learn operators corresponding to the numerical schemes \eqref{eqn:sing_pic_iter}, \eqref{eqn:pic_iter}, and \eqref{eqn:newton}.
We restrict our study to first order implicit methods as higher order RK methods have a more cumbersome representation, and explicit methods lack certain stability properties. 

Multi-input networks that also take in reaction functions as input, offer the advantage of predicting the time evolution for a whole class of reaction-diffusion equations without retraining. This requires an appropriate encoding for the reaction function. 

\subsection{Encoding the Reaction Function} A natural way to encode the reaction function $f$ in a finite dimensional space is with a basis encoder. For example, using the coefficients of a truncated Taylor series around the origin:
\begin{equation} \label{eqn:basis}
    f(x) = \sum_{i=0}^{p-1} a_i T_i(x) + \mathcal{O}(x^{p}), \quad \text{where } T_i(x) = x^i.
\end{equation}
This representation is exact in many important models like the Newell–Whitehead-Segel equation \cite{newell1969finite} and Fisher-KPP equation \cite{kolmogorov1991study}, with a polynomial reaction function. 

For periodic H\"older reaction functions, alternatively, we may use a truncated Fourier expansion of $f \in L^2[-1,1]$ given by $f(x) \approx \sum_{i=0}^{p-1} a_i T_i(x)$, with:
\begin{align*}
    T_0(x) = \frac{1}{2}, \quad T_{2p-1}(x) = \sin(p\pi x), \quad T_{2p}(x) = \cos(p\pi x) \text{ for $p\geq1$}.
\end{align*}
The reaction function is then encoded as the vector $\mathrm{a} = \begin{bmatrix} a_0,a_1,\cdots, a_{p-1} \end{bmatrix}^\top $.

Next, we study the time-stepping solvers defined in \eqref{eqn:pic_iter}, \eqref{eqn:sing_pic_iter}, and \eqref{eqn:newton} and present our main estimates on the generalization error for learning these mappings.

\subsection{Generalization Error}
We treat the implicit time-stepping operators with Picard and Newton solvers separately. A preliminary lemma is established below.

\begin{lemma} \label{lemma:contract}
Let $L$ be the solution operator for $[I-\Delta t\Delta]u^1=u^0$ for $u^0 \in C_0(\bar{D})$ with homogeneous Dirichlet boundary conditions on smooth $\partial D$. Then, for $\Delta t> 0$,
\[
\|Lu^0-Lv^0\|_\infty \leq \|u^0-v^0\|_\infty, \ \forall \ u^0,v^0 \in C_0(\bar{D}).
\]
 \end{lemma}

\begin{proof}
Let $u^1 =Lu^0$ and $v^1 = Lv^0$ for $u^0,\ v^0 \in C_0(\bar{D})$. The Laplacian associated with the heat equation induces a heat semi-group, $\{T(t)u\}_{t\geq 0}$ for $u \in C_0(\bar{D})$. Using semi-group theory \cite{abadias2022asymptotic, engel2000one}, we write:
\begin{align}
   \frac{1}{\Delta t}\left[\frac{1}{\Delta t} -\Delta\right]^{-1}u^0(x) = \frac{1}{\Delta t } \int_{0}^{\infty} e^{-t/\Delta t} T(t) u^0 \, dt =: T_{\Delta t} (u^0). \label{eqn:heat_sg}
\end{align}

As $\{T(t)u^0\}_{t\geq 0}$ and $\{T(t)v^0\}_{t\geq 0}$ share the same boundary conditions, the strong maximum principle for the heat equation in bounded domains implies that
\(
    \|T(t) (u^0-v^0)\|_\infty \leq  \|u^0-v^0\|_\infty.
\)
This gives:
\[
\|u^1-v^1\|_\infty \leq  \left\lvert \left\lvert  \frac{1}{\Delta t} \int_{0}^{\infty} e^{-t/\Delta t} T(t)\left(u^0-v^0\right) \, dt \right\rvert \right\rvert_\infty
\leq \frac{\|u^0-v^0\|_\infty}{\Delta t}  \int_{0}^{\infty} e^{-t/\Delta t} \, dt.
\]
The integral term evaluates to $\Delta t$, proving the non-expansiveness of the operator.

\end{proof}

\subsubsection{Implicit Euler with Picard's method} We start by considering the multi-input method \eqref{eqn:pic_iter}. First, we establish that the time-stepping operator $\Psi_{\text{P}}$ is Lipschitz continuous. Since $\Omega_p$ in Assumption \ref{assump:compact_supp} is a compact subset of $C^1(\mathbb{R})$, there exists $L_p \in \mathbb{R}$ such that: 
\begin{equation} \label{eqn:unif_lip}
    |f(x) - f(y)| \leq L_p |x-y|_\infty \ \forall \ x,y \in \mathbb{R}, \ \forall \ f \in \Omega_p.
\end{equation}

\begin{proposition}  \label{prop:pic_lipschitz}  
        Suppose Assumption \ref{assump:compact_supp} holds, and $\Delta t L_p < 1$. The operator \(\Psi_{\text{P}}:  \Omega_0 \times \Omega_p \subset C_0(\bar D) \times C^1(\mathbb{R}) \to C_0(\bar D) ,\)
        defined in \eqref{eqn:pic_iter}, is Lipschitz continuous. Specifically, for $ (u^0,f),\ (v^0,g)\in \Omega_0\times \Omega_p$, we have:
 \begin{align}
 \|\Psi_{\text{P}}(u^0,f) - \Psi_{\text{P}}(v^0,g)\|_\infty \leq \frac{\max\{1,\Delta t\}}{1-L_p \Delta t}\left[\|u^0-v^0\|_\infty +  \|f-g\|_{C^1} \right].
 \label{eqn:lipschitz}
 \end{align}

\end{proposition}

\begin{proof} 
Let $u^{(i)}= \Phi_{\text{BE}}(u^{(i-1)},u^0,f),\ v^{(i)}(x) = \Phi_{\text{BE}}(v^{(i-1)}, v^0,g)$ for $1\leq i \leq m$,  denote the intermediate iterates, as in \eqref{eqn:implicit}. The last iteration is written in terms of the semi-group operator as:
\begin{equation}
    u^{(m)}(x) = T_{\Delta t} (u^0) + \Delta t T_{\Delta t} (f(u^{(m-1)})), \quad \text{for } m>0.
\end{equation}
Using triangle inequality and linearity of the semi-group operator:
\begin{align}
    \|u^{(m)} -v^{(m)}\|_\infty &\leq \|  T_{\Delta t} (u^0) - T_{\Delta t} (v^0)\|_\infty + \Delta t \| T_{\Delta t} (f(u^{(m-1)}))-  T_{\Delta t} (g(v^{(m-1)}))\|_{\infty} \nonumber \\
     &\leq \|T_{\Delta t}( u^0- v^0)\|_\infty +\Delta t \|T_{\Delta t} [ f(u^{(m-1)})-   g(u^{(m-1)})]\|_{\infty} \nonumber \\ &+\Delta t \| T_{\Delta t}[ g(u^{(m-1)})-  g(v^{(m-1)})]\|_{\infty}.
     \label{eqn:lip1}
\end{align}
Then, using Lemma \ref{lemma:contract}, and the uniform Lipschitz property in \eqref{eqn:unif_lip}:
\begin{align*}
      \|u^{(m)}-v^{(m)}\|_\infty \leq \|u^0 - v^0\|_\infty + \Delta t \|f-g\|_\infty + \Delta t L_p \|u^{(m-1)} - v^{(m-1)}\|_\infty.
\end{align*}
The required bound follows by recursively applying the above argument, giving:
 \begin{align*}
     \|u^{(m)}-v^{(m)}\|_\infty &\leq \sum_{i=0}^{m}(\Delta t L_p)^i \left[\|u^0 -v^0\|_\infty + \Delta t \|f-g\|_\infty\right]\\
    & \leq \max\{1,\Delta t\}[{1-\Delta t L_p}]^{-1} \left[\|u^0 -v^0\|_\infty + \|f-g\|_\infty\right] \text{ if } \Delta t L_p < 1.
 \end{align*}
Since $\|f-g\|_\infty \leq \|f-g\|_{C^1}$ by definition, we get the required result.
\end{proof}

\begin{remark}
    The step-size restriction $\Delta t L_p < 1$ resembles the stability condition for explicit methods. However, this is a reasonable assumption as it ensures that the fixed point map is a contraction, guaranteeing the convergence of this method. 
\end{remark}

Next, we demonstrate the low complexity structure of this operator. Specifically, we aim to represent it as a composition of functions, each involving a linear transformation followed by a low-dimensional non-linearity. These composing functions can be interpreted as the layers of a neural network.

\begin{assumption} \label{assump:pic_encoder}
    The initial condition $u^0 \in C_0(\bar{D})$ is encoded via discretization encoder to $\mathbb{R}^{d_{\Xcal_1}}$.
\end{assumption}

\begin{assumption} \label{assump:pic_r_encoder}
     The reaction function $f \in C^1(\mathbb{R})$ is encoded via basis encoder \eqref{eqn:basis} to $\mathbb{R}^{d_{\Xcal_2}}$.
\end{assumption}

\begin{proposition}  \label{lem:pic_lc}
  Suppose Assumptions \ref{assump:pic_encoder}-\ref{assump:pic_r_encoder} hold. Consider $\Psi_{\text{P}}:  C_0(\bar{D}) \times C^1(\mathbb{R})\to  C_0(\bar{D})$ defined in \eqref{eqn:pic_iter}. This operator exhibits the low complexity structure in Definition~\ref{defn:low_comp}, with $k=2m$, $d_\text{max} = 2$, and $\ell_\text{max} = (d_{\Xcal_2}+1)d_{\Xcal_1} + d_{\Xcal_2}$.
\end{proposition}

\begin{proof}
The operator $\Psi_{\text{P}}$ consists of $m$ Picard iterations, defined in \eqref{eqn:pic_iter}. We argue that this operator possesses the low complexity structure with each iteration $i$ ($1\leq i\leq m)$ corresponding to two layers of a network.

We demonstrate this for the first iteration. Let $(u^0,f)\in C_0(\bar{D}) \times C^1(\mathbb{R})$. Denote the encoded inputs $\mathrm{u}^0 = E_{\Xcal_1}(u^0) \in \mathbb{R}^{d_{\Xcal_1}}$, and $\mathrm{a} = E_{\Xcal_2}(f) = [a_k]_{k=0}^{p-1} \in \mathbb{R}^{d_{\Xcal_2}}$ (for $p=d_{\Xcal_2}$). The state vector entering the first iteration is $z^{(0)} = [\mathrm{u}^{0}; \mathrm{a}] \in \mathbb{R}^{d_{\Xcal_1} + d_{\Xcal_2}}$.

The first layer $G^{1}$ computes the individual basis terms $p_{j,k} = a_k T_k(\mathrm{u}^{0}_j)$ for $j\in \{1, \ldots, d_{\Xcal_1}\}, k\in \{0, \ldots, p-1\}$, needed for the reaction function approximation $P(\mathrm{u}^{0}; \mathrm{a})= \sum_{k=0}^{p-1} a_k (\mathrm{u}_j^{0})^k$. Linear projections $V^{1}_{j,k} \in \mathbb{R}^{1\times d_\Xcal}$, act on the input state $z^{0}$ to extract the necessary pair $(\mathrm{u}^{0}_j, \mathrm{a}_k)$ for each term. Specifically, the projection $(V^{1}_{j,k})^T z^{(0)}$ selects the $j$-th component from the $\mathrm{u}^{0}$ block and the $k$-th coefficient from the $\mathrm{a}$ block. A corresponding two-dimensional non-linearity $g_{j,k}(x, y) = yT_k(x)$  is applied to this pair $(\mathrm{u}^{0}_j, \mathrm{a}_k)$ to compute $p_{j,k}$. This layer must also retain $\mathrm{u}^0$ and $\mathrm{a}$ for computing subsequent iterates. This is achieved with projections extracting $\mathrm{u}^0$ and $\mathrm{a}$ from $z^{(0)}$ followed by identity functions ($g(v)=v$). The full output of layer $G^{1}$ is the intermediate vector $z' = [\mathrm{p}; \mathrm{u}^0; \mathrm{a}]$, where $\mathrm{p} = [p_{j,k}]$. The dimension of $z'$ is $\ell_1 = d_{\Xcal_1}d_{\Xcal_2} + d_{\Xcal_1} + d_{\Xcal_2}$.

The next layer $G^{2}$ performs a linear update, taking $z'$ as input and outputs the state for the next iteration, $z^{(1)} = [\mathrm{u}^{(1)}; \mathrm{u}^0; \mathrm{a}]$. Let $L$ represents the numerical integration weights of the linear operator $[1-\Delta t \Delta]^{-1}$. Computing $\mathrm{u}^{(1)} = L(\mathrm{u}^0 + \Delta t P(\mathrm{u}^{0}; \mathrm{a}))$ from $z'$ involves summing components of $\mathrm{p}$ (to get $P(\mathrm{u}^{0}; \mathrm{a})$), scaling by $\Delta t$, adding $\mathrm{u}^0$ (extracted from $z'$), and finally, applying the matrix $L$. Additionally, $\mathrm{u}^0$ and $\mathrm{a}$ should be retained. As all these steps are linear operations on $z'$, the required maps for computing each component of $\mathrm{u}^{(1)}$ and passing through $\mathrm{u}^0$ and $\mathrm{a}$ are encoded in matrices $V^{2}_s$ for $1 \leq s \leq 2d_{\Xcal_1} + d_{\Xcal_2} $. This layer uses only identity functions, so that the $s$-th output component is $(V^{2}_s)^T z'$. The output dimension is $\ell_2 = 2d_{\Xcal_1} + d_{\Xcal_2}$.

Extrapolating this two-stage process for $m$ iterations demonstrates the low complexity structure of the operator. The total number of layers is $k=2m$. The overall maximum non-linearity input dimension is $d_{\max} =  2$. The maximum intermediate dimension is $\ell_{\max} = \ell_1 = (d_{\Xcal_2}+1)d_{\Xcal_1} + d_{\Xcal_2}$ [see Figure~\ref{fig:lowcomp}].
\end{proof}

\vspace{-2mm}
\begin{figure}[htbp]
\centering
\includegraphics[width=\linewidth]{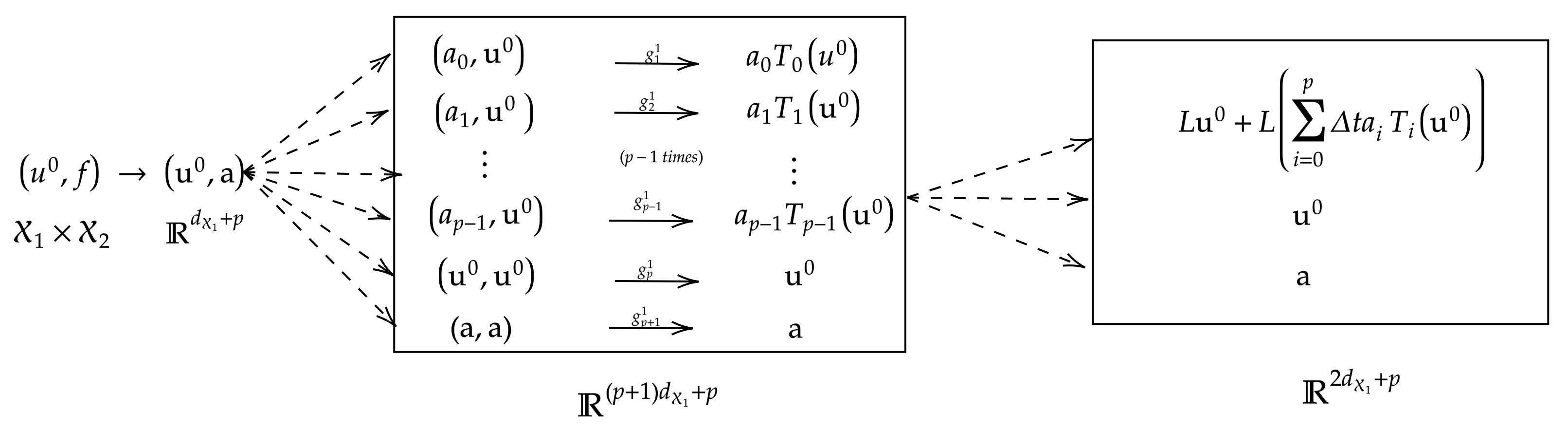}
\vspace{-6mm}
\caption{Schematic representation of the first iteration in Picard's method. Dashed lines indicate linear transformations while full lines indicate non-linearities. In the first layer, $p$ tuplets $(a_i,\mathrm{u}^0)$ for $1\leq i \leq  p$ are obtained, after which the nonlinearity $g^1_i(x,y) = x T_i(y)$ is applied to get $a_iT_i(\mathrm{u}^0)$. Simple linear transformations then give the first iterate of this method.}
  \label{fig:lowcomp}
\end{figure}%
\vspace{-2mm}

As $\Ycal= \Xcal_1$, the same encoder-decoder pairs may be used for both spaces, and by construction $d_\Ycal = d_{\Xcal_1}$. Then, we obtain a bound on the generalization error.

\begin{theorem} \label{thm:gen_err_pic}
Suppose Assumptions \ref{assump:compact_supp} and \ref{assump:pic_encoder}-\ref{assump:pic_r_encoder} hold. Let $\GNN$ be the minimizer of the optimization (\ref{eqn:optimization}) where the target operator $\Psi_{\text{P}}: \Omega_0\times \Omega_p \to \Omega_1$ is defined in \eqref{eqn:pic_iter}. The network architecture $\mathcal{F}_\mathrm{NN}(d_{\Xcal_1}+ d_{\Xcal_2},2mL,{p},M)$ is as defined in (\ref{eqn:FNN}) with parameters:
\begin{equation*}
\begin{aligned}
		Lp =  \Omega \left( (d_{\Xcal_1}+d_{\Xcal_2})^{\frac{1}{4}} n^{\frac{1}{4}} \right)
  \,, M \geq \sqrt{\ell_{\text{max}}} L_{E_{\Ycal}} R_{\Ycal} \,, \text{ for } \ell_\text{max} = (d_{\Xcal_2}+1)d_{\Xcal_1} +d_{\Xcal_2},
\end{aligned}
\end{equation*}	
Define $L_{{P}} =\max\{1,\Delta t\}/(1-L_p \Delta t)$ assuming $L_p \Delta t<1$. Then:
	\begin{equation*}
		\begin{aligned}
\mathcal{E}_{\text{gen}}\lesssim  L_{{P}}^{2} \log(L_{{P}})(\ell_{\text{max}})^{5/2} n^{-\frac{1}{2}} \log n + \mathcal{E}_{\text{proj}} \,.
		\end{aligned}
	\end{equation*}
	The constants in $\lesssim$ and $\Omega(\cdot)$ depend only on $m,p,\ell_\text{max},R_\Ycal,L_{E_{{\Xcal_1}\times {\Xcal_2}}}$ and $L_{D_{{\Xcal_1}\times {\Xcal_2}}}$. 
\end{theorem}
\begin{proof}
     The theorem follows as a consequence of Theorem \ref{thm:low_complexity}, using the Lipshcitz property in Proposition \ref{prop:pic_lipschitz} and the low complexity structure demonstrated in \ref{lem:pic_lc}.
\end{proof}

The key take-away from Theorem \ref{thm:gen_err_pic} is the polynomial dependence on $d_{\Xcal_1 \times \Xcal_2}$ implying that this neural operator is free of CoD, unlike traditional methods. The result, unsurprisingly, implies that a smaller time-step lowers the generalization error.

\begin{remark}
    The projection error $\mathcal{E}_{proj}$ depends on the basis expansion chosen for the reaction function. A higher encoding dimension lowers the projection error while increasing the complexity, and there remains a need to optimize this trade-off.
\end{remark}

Adapting the proof techniques from Propositions \ref{prop:pic_lipschitz}–\ref{lem:pic_lc} yields a corresponding generalization error bound for the single-input target operator $\Phi_{\text{P}}:\Omega_0\to C_0(D)$. The encoding dimension $d_{\Xcal}$ is thus $d_{\Xcal_1}$.

\begin{theorem}
    Suppose Assumptions \ref{assump:compact_supp} and \ref{assump:pic_encoder} hold. Let $f\in C^1(\mathbb{R})$ be $L_f-$ Lipschitz continuous. Let $\GNN$ be the minimizer of the optimization (\ref{eqn:optimization}) where the target operator $\Phi_{f,P}: \Omega_0 \to C_0(\bar D)$ is defined in \eqref{eqn:pic_iter}. The network architecture $\mathcal{F}_\mathrm{NN}(d_{\Xcal},2mL,{p},M)$ is as defined in (\ref{eqn:FNN}) with parameters:
\begin{equation*}
\begin{aligned}
		Lp =  \Omega \left( (d_{\Xcal}^{\frac{1}{2}} n^{\frac{1}{6}} \right)
  \,, M \geq \sqrt{\ell_{\text{max}}} L_{E_{\Ycal}} R_{\Ycal} \,, \text{ for } \ell_\text{max} =2d_{\Xcal} ,
\end{aligned}
\end{equation*}	
Define $L_{{P}} =1/(1-L_f \Delta t)$ assuming $L_f\Delta t<1$. Then:
	\begin{equation*}
		\begin{aligned}
\mathcal{E}_{\text{gen}}\lesssim  L_{{P}}^{2} \log(L_{{P}})(\ell_{\text{max}})^{3} n^{-\frac{2}{3}} \log n + \mathcal{E}_{\text{proj}} \,.
		\end{aligned}
	\end{equation*}
	The constants in $\lesssim$ and $\Omega(\cdot)$ depend only on $m,p,\ell_\text{max},R_\Ycal,L_{E_{{\Xcal}}}$ and $L_{D_{{\Xcal}}}$. 
\end{theorem}
\begin{proof}
    The proof follows the same idea as Theorem \ref{thm:be_para}. The operator is Lipschitz continuous with constant $1/(1-L_f \Delta t)$. This estimate holds by setting $g=f$ in Proposition \ref{prop:pic_lipschitz}. For the low complexity structure, the only non-linearity the network learns is the one-dimensional non-linearity of the reaction function $f$. The maximal width is from computing intermediate values $[\mathrm{u}^0;f(\mathrm{u}^{(i)}] \in \mathbb{R}^{2d_\Xcal}$ for $\mathrm{u}^0= E_\Xcal(u^0), \mathrm{u}^{(i) }=E_\Xcal(u^{(i)})$. This structure reveals $\ell_{max}=2d_\Xcal$ and $d_{\max }= 1$ giving the desired result.
\end{proof}

In practice, the stepsize restriction of Picard's method ruins the nice stability properties of the implicit method. Newton's method offers an alternative and its effectiveness for stiff ODEs is well-studied \cite{liniger1970efficient}.

\subsubsection{Implicit Euler with Newton's method}
For simplicity, we focus on the single-input case for the operator corresponding to this algorithm. However, the arguments used here easily extend to the multi-input scenario.

 \begin{assumption}\label{assump:lipchitz_reaction}
     The reaction function $f \in C^1(\mathbb{R})$ is such that $f$, $f'$ are Lipschitz with Lipschitz constants $L_f, L_{f'}>0$, respectively.
 \end{assumption}

  \begin{assumption} \label{assump:decreasing_reaction}
The reaction function $f$ is decreasing, i.e., $f'(x)\leq 0, \ \forall \ x \in \mathbb{R}$.
 \end{assumption}

 \begin{lemma} \label{lemma:newton_kant}
     Suppose Assumptions \ref{assump:compact_supp}, \ref{assump:lipchitz_reaction}-\ref{assump:decreasing_reaction} hold. Then, for all $u^0\in \Omega_0$,
     \begin{enumerate}[label=(\alph*)]
         \item  The operator $\Phi_{u^0}:C_0(\bar{D}) \to C_0(\bar{D})$ defined in \eqref{eqn:root_newt} is differentiable.
         \item For all $w \in C_0(\bar{D})$, $D\Phi_{u^0}(w)$ is invertible, and satisfies $\|[D\Phi_{u^0}(w)]^{-1}\|_\infty \leq 2$.
         \item For all $ w_1, w_2 \in C_0(\bar{D})$, the following bound holds:
         \begin{align}
         \label{eqn:inv_lip} \|D\Phi_{u^0}(w_1) - D\Phi_{u^0}(w_2)\|_\infty \leq L_{f'} \Delta t \|w_1 -w_2\|_\infty\ .
     \end{align}
     \end{enumerate}
     
 \end{lemma}
\begin{proof}
    Recall the definition of the operator in \eqref{eqn:root_newt}:
\[
\Phi_{u^0}(w) := w- [1-\Delta t \Delta ]^{-1} {u^0} - [1-\Delta t \Delta ]^{-1} \Delta t f(w).
\]
To prove (a), we compute its Fr\'echet derivative's action on $h \in C_0(\bar{D})$ as: 
\begin{align}
  D\Phi_{u^0}(w)h &= [I-D\Psi(w)]h \label{eqn:frechet},\\ \text{ where } \Psi(w) = [1-\Delta t \Delta]^{-1}(\Delta t f(w)), &\text{ and } 
  (D\Psi(w))h =  [1-\Delta t \Delta]^{-1}(\Delta t f'(w)h), \nonumber
\end{align}
Notably, the Fr\'echet derivative $D\Phi_{u^0}(w)$ is independent of $u^0$. Observe that the operator norm $\|D\Phi_u(w)\|_\infty$ for any $u,\ w \in C_0(\bar{D})$ is bounded:
\begin{align}
    \|D\Phi_u(w)\|_\infty &= \|I - D\Psi(w)\|_\infty \leq 1 + \sup_{\substack{h\in C_0(\bar{D}),\\
    \|h\|_\infty=1}} \|(1-\Delta t \Delta)^{-1}[ \Delta t f'(w)h]\|_\infty  \nonumber \\
    &\leq 1+\Delta t \|f'\|_\infty   \leq 1+\Delta t L_f, \label{eqn:grad_bound}
\end{align} 
where we have used the non-expansiveness of $[1-\Delta t \Delta]^{-1}$ (refer Lemma \ref{lemma:contract}), and the Lipschitz property of the reaction function. 

For Dirichlet boundary conditions, the negative Laplacian $-\Delta$ is positive-definite. Assumption \ref{assump:decreasing_reaction} then implies that $D\Psi$ is a negative-definite operator. Hence, $D\Phi_u(w)$ $\forall \ u,\ w \in C_0(\bar{D})$ is positive-definite and consequently, invertible. 

To prove the bound in (b), let $D\Phi_u(w)h = g$ for $h \in C_0(\bar{D})$. 
Define an auxiliary function $v(x) = h(x) - g(x)$. By the definition of $D\Phi_u(w)$:
\[
v = h - g = [I - \Delta t \Delta]^{-1}(\Delta t f'(w)h).
\]
Applying the operator $(I - \Delta t \Delta)$ to both sides yields the equation:
\begin{align*}
    (I - \Delta t \Delta)v &= \Delta t f'(w)h 
\implies v - \Delta t \Delta v = \Delta t f'(w)(v+g).
\end{align*}
First, we claim that $\sup_{\bar{D}} v \le \|g\|_\infty$. Let $v$ attain its maximum at a point $x_0 \in \bar{D}$. If $x_0 \in \partial D$, then since $h(x_0) = 0$, we have $v(x_0) = -g(x_0) \le \|g\|_\infty$ proving our claim.
     
Else, suppose $x_0 \in D$ and $v(x_0)>0$.  Evaluating the PDE at $x_0$, we get:
    \[
    {v(x_0) - \Delta t \Delta v(x_0)} = \Delta t f'(w(x_0))(v(x_0)+g(x_0)).
    \]
    As $x_0$ is an interior maximum, $\Delta v(x_0) \le 0$, and hence the term on the left is positive. Since $\Delta t > 0$ and $f'(w) \le 0$, the right-hand side can only be positive if $v(x_0)+g(x_0) < 0$, or in other words, $v(x_0)  < -g(x_0) \le \|g\|_\infty.$

Using a similar argument, we can show that $\inf_{\bar{D}} v \ge -\|g\|_\infty$, which establishes that for all $x \in \bar{D}$, $-\|g\|_\infty \le v(x) \le \|g\|_\infty$. This implies:
\[
\sup_{\bar{D}} h = \sup_{\bar{D}}(v+g) \le \sup_{\bar{D}} v + \sup_{\bar{D}} g \le \|g\|_\infty + \|g\|_\infty = 2\|g\|_\infty.
\]
and similarly, $\inf_{\bar{D}} h \ge -2\|g\|_\infty$.
Consequently, $\|h\|_\infty \le 2\|g\|_\infty$, and hence,
\begin{align}
    \|[D\Phi_u(w)]^{-1}\|_\infty \leq 2 \ \forall \ u,\ w \in C_0(\bar{D}).\label{eqn:inv_grad_bd}
\end{align}
Finally, we need to demonstrate \eqref{eqn:inv_lip}. For $w_1, w_2 \in C_0(\bar{D})$, applying \eqref{eqn:frechet}:
\begin{align*}
    \|D \Phi_{u^0}(w_1)) - D\Phi_{u^0}(w_2)\|_\infty &\leq  \sup_{\substack{h\in C_0(\bar{D}),\\
    \|h\|_\infty=1}} \|[1-\Delta t \Delta]^{-1}[\Delta t (f'(w_1)-f'(w_2))h]\|_\infty \nonumber \\
    &\leq L_{f'} \Delta t \|w_1 -w_2\|_\infty, 
\end{align*}
Here, we have again used the contractivity in Lemma \ref{lemma:contract}, along with Assumption \ref{assump:lipchitz_reaction}. This proves the result.
\end{proof}

By Lemma \ref{lemma:newton_kant},  $D\Phi_{u}(w), \forall \  u, w \in C_0(\bar{D})$ is a bounded linear operator.  If invertibility is assumed, $\|(D\Phi_{u^0}(w))^{-1}\|_\infty$ is bounded by the Bounded Inverse Theorem. But an explicit bound on the norm is difficult to obtain, motivating Assumption \ref{assump:decreasing_reaction}.

\begin{remark}
The condition $\Delta t L_f < 1$ also ensures invertibility of $D\Phi_{u^0}(w)$, with the Neumann series expansion of the inverse providing a bound similar to \eqref{eqn:inv_grad_bd}. This stability restriction, however, is often not set for Newton's methods in practice. 
\end{remark}

  As a consequence of Lemma \ref{lemma:newton_kant}, Newton's method converges, provided the initial iterate is sufficiently close to the exact solution.

\begin{assumption}\label{assump:init_iterate}
  The input-space $\Omega_0$ satisfies the constraint, \[L_{f'}\Delta t \sup_{u^0\in \Omega_0}\|\Phi_{\text{BE}}(u^0)-u^0\|_\infty \leq \sqrt{2}-1.\]
\end{assumption}
Note that $\|\Phi_{\text{BE}}(u^0)-u^0\|_\infty$ is $\mathcal{O}(\Delta t)$ from \eqref{eqn:implicit}. 
Under Assumptions \ref{assump:lipchitz_reaction}-\ref{assump:init_iterate}, Newton's method converges by the Kantorovich theorem \cite{liniger1970efficient}.

\begin{proposition} \label{lemma:newt_lipschitz} 
      Suppose Assumptions \ref{assump:compact_supp} and  \ref{assump:lipchitz_reaction}-\ref{assump:init_iterate}  hold. The operator $\Phi_{\text{N}}:  \Omega_0 \subset (C_0(\bar{D}),\|\cdot\|_\infty) \to (C_0(\bar{D}),\|\cdot\|_\infty)$, defined in \eqref{eqn:newton}, is Lipschitz with:
     
     \begin{align}
     \|\Phi_{\text{N}}(u^0)- \Phi_{\text{N}}(v^0)\|_\infty \leq L_N\|u^0-v^0\|_\infty
     \label{eqn:n_lipschitz},
     \end{align}
     where $L_N = K^m + \frac{2(K^m-1)}{K-1}$ with $K = 3 + 2\Delta t L_f + \mathcal{O}(\Delta t^{2})$ for $m$ Newton iterations. The constants in the term $\mathcal{O}(\Delta t^{2})$ only depend on $\Omega_0,D$, and $f$. 
\end{proposition}
\begin{proof}

Let $u^1 = \Phi_{\text{N}}(u^0)$, and $v^1 = \Phi_{\text{N}}(v^0)$ with $u^{(i)}$ and $v^{(i)},\ i = 1, \ \dots ,\ m$ as their respective intermediate steps. The last step of the iterative process given in \eqref{eqn:newton}, can be bounded above using the triangle inequality as, 
\begin{align}
    \|&\Phi_{\text{N}}(u^0)-\Phi_{\text{N}}(v^0)\|_\infty 
    \leq \|u^{(m-1)} -v^{(m-1)}\|_\infty  \nonumber\\&
    \qquad \hspace{1.3cm} + \|[D\Phi_{u^0}(u^{(m-1)})]^{-1}\Phi_{u^0}(u^{(m-1)}) -[D\Phi_{v^0}(v^{(m-1)})]^{-1}\Phi_{v^0}(v{^{(m-1)}})\|_\infty \nonumber \\
    &\leq \|u^{(m-1)} -v^{(m-1)}\|_\infty  + \underbrace{\|[D\Phi_{u^0}(u^{(m-1)})]^{-1}\|_\infty  \cdot \|\Phi_{u^0}(u^{(m-1)}) -\Phi_{v^0}(v^{(m-1)}) \|_\infty}_{\rm I}
    \nonumber \\
    & + \underbrace{\| \Phi_{v^0}(v^{(m-1)})\|_\infty \cdot \|[D\Phi_{u^0}(u^{(m-1)})]^{-1} - [D\Phi_{v^0}(v^{(m-1)})]^{-1}\|_\infty}_{\rm II}
    \label{eqn:lip_step}
\end{align}
We bound each term separately. 
Applying \eqref{eqn:inv_grad_bd} along with the triangle inequality, 
\begin{align}
    {\rm I} &\leq   2\|\Phi_{u^0}(u^{(m-1)}) -\Phi_{u^0}(v^{(m-1)}) \|_\infty + 2\|\Phi_{u^0}(v^{(m-1)}) -\Phi_{v^0}(v^{(m-1)}) \|_\infty.
    \label{eqn:Isplit}
\end{align}
The first term in \eqref{eqn:Isplit} can be bounded using the mean value inequality for Banach spaces and for the second term, we use the definition of $\Phi_{u^0}$. Consequently,
\begin{align}
   {\rm I}  
   &\leq  2\max_{w \in C_0(\bar{D})}\|D \Phi_{u^0}(w)\|_\infty \cdot  \|u^{(m-1)}-v^{(m-1)}\|_\infty + 2\|[1-\Delta t \Delta]^{-1}[u^0-v^0]\|_\infty, \nonumber \\
   &\leq 2(1+ \Delta t L_f) \|u^{(m-1)}-v^{(m-1)}\|_\infty + 2\|u^0-v^0\|_\infty. \label{eqn:term1}
\end{align}
To obtain the last inequality, we use  \eqref{eqn:grad_bound} along with Lemma \ref{lemma:contract}.  Next, we need to bound $\rm II$ in \eqref{eqn:lip_step}. Expanding $\Phi_{u^0}$ around its root $u^1$ using a Taylor expansion:
\begin{align}
    \|\Phi_{u^0}(u^{(m-1)})\|_\infty &\leq \|\Phi_{u^0}(u^1)\|+ \|D\Phi_{u^0}(u^1)\|\cdot \|u^1-u^{(m-1)}\| + \mathcal{O}(\|u^1-u^{(m-1)}\|^2) \nonumber \\
    &= 0 + \mathcal{O}\left(\|u^1-u^0\|^{2^{m-1}}\right)  = \mathcal{O}\left(\Delta t^{2^{m-1}}\right). \label{eqn:phiu_estimate}
\end{align}
Here we have used the quadratic convergence of Newtons method along with the fact that $\|u^1-u^0\| = \mathcal{O}(\Delta t)$ (from \eqref{eqn:implicit}), where the hidden constants depend on $u^0,\ f,$ and $D$.  Next, observe that:
\begin{align}
    &\|[D\Phi_{u^0}(u^{(m-1)})]^{-1} - [D\Phi_{v^0}(v^{(m-1)})]^{-1}\|_\infty \nonumber \\
    \leq& \|[D\Phi_{u^0}(u^{(m-1)})]^{-1}\|_\infty \cdot 
    \|D \Phi_{u^0}(u^{(m-1)}) - D\Phi_{v^0}(v^{(m-1)})\|_\infty \cdot
    \|[D\Phi_{v^0}(v^{(m-1)})]^{-1}\|_\infty \nonumber \\
    \leq & 4 \|D \Phi_{u^0}(u^{(m-1)}) - D\Phi_{v^0}(v^{(m-1)})\|_\infty \quad \text{(from \eqref{eqn:inv_grad_bd})}.
\end{align}
Putting this together with \eqref{eqn:inv_lip} and  \eqref{eqn:phiu_estimate}, an upper bound on $\rm II$ may be obtained:
\begin{align}
  {\rm II}& \leq \mathcal{O}(\Delta t^{2^{m-1}})\|D \Phi_{u^0}(u^{(m-1)}) - D\Phi_{v^0}(v^{(m-1)})\|_\infty \nonumber\\
  &\leq \mathcal{O}(\Delta t^{2^{m-1}+1})\|u^{(m-1)} -v^{(m-1)}\|_\infty \label{eqn:term2}
\end{align}

The last inequality is from \eqref{eqn:inv_lip} as $D \Phi_{u^0}$ is independent of $u^0$. Combining \eqref{eqn:term1} and \eqref{eqn:term2} with equation \eqref{eqn:lip_step}, we get,
\[
     \|\Phi_{\text{N}}(u^0)-\Phi_{\text{N}}(v^0)\|_\infty \leq [3+ 2\Delta t L_f + \mathcal{O}(\Delta t ^{2^{m-1}+1})]\|u^{(m-1)}-v^{(m-1)}\|_\infty + 2\|u^0-v^0\|_\infty. 
\]
Proceeding recursively gives the required result:
\begin{align*}
    \|\Phi_{\text{N}}(u^0)-\Phi_{\text{N}}(v^0)\|&\leq L_N\|u^0-v^0\|.
\end{align*}
where $L_N = K^m + \frac{2(K^m-1)}{K-1}$ with $K = 3 + 2\Delta t L_f + \mathcal{O}(\Delta t^{2})$. 
\end{proof}
\begin{remark}
    Assumption \ref{assump:init_iterate} is satisfied when $\Delta t L_f (R_\Xcal + L_NR_\Xcal) \leq \sqrt{2}-1 $.
\end{remark}

\begin{remark}
    Modified Newton's method is often used instead of the full method to reduce the computational cost of the LU-decomposition at each iteration. Proposition \ref{lemma:newt_lipschitz} can be easily extended for the modified method, with the proof following the same structure but with slight modifications to the analysis of term II in \eqref{eqn:lip_step}.
\end{remark}

Next, we demonstrate the low complexity structure of this algorithm. Newton's method requires solving a linear system and we first show that this step of the solver has low complexity.

\begin{lemma}\label{lemma:LU_lc}
Consider the solution operator $\Psi: \mathbb{R}^{d_\Xcal \times d_\Xcal} \times \mathbb{R}^{d_\Xcal}\ni (A,b) \mapsto x \in \mathbb{R}^{d_\Xcal}$ for the linear system $Ax=b$, where $A \in \mathbb{R}^{d_\Xcal \times d_\Xcal}$ is invertible, and $x \in \mathbb{R}^{d_\Xcal}$ is computed using the LU decomposition algorithm. This operator has the low complexity structure in Definition \ref{defn:low_comp}, with $k = 6d_\Xcal,\ d_{\text{max}} = 2$ and $\ell_{\text{max}} = \mathcal{O}(d_\Xcal^2)$.
\end{lemma}

\begin{proof}
The proof follows the three stages of the algorithm: LU decomposition of $A$, forward substitution to solve $Ly=b$, and backward substitution to solve $Ux=y$. Each stage is unrolled into the low-complexity structure. The initial input is $z_0 = [\text{vec}(A); b] \in \mathbb{R}^{d_\Xcal^2 + d_\Xcal}$, which is appended with the outputs of each successive layer.

The LU decomposition stage can be unrolled into $d_\Xcal$ steps, where each step $j = 1, \cdots, d_\Xcal$ corresponds to two layers. Let $[L]_{j-1},[U]_{j-1}$ denote the elements of $L$ and $U$ computed up to the $j-1$th stage. The first layer $G^{2j-1}$ takes as input $[z_0;[L]_{j-1},[U]_{j-1}] \in \mathbb{R}^{d_j}$. For $i<j$ and $k \geq j$, $G^{2j-1}$ computes all products of the form $l_{ki}u_{ik}$. The corresponding projection matrix $V^{2j-1}_k \in \mathbb{R}^{d_{j} \times 2}$ selects the two appropriate components for the product from the input. The product terms are then computed using the non-linearity $g(x, y) = xy$.  The number of such products is approximately $2(d_\Xcal-j)(j-1)$, leading to a maximal layer width $\ell_{\text{max}} = \mathcal{O}(d_\Xcal^2)$. The output of this layer is thus $[z_0;\ [L]_{j-1};\ [U]_{j-1};\  \{l_{ji}u_{ik}\}_{i=1, k=j}^{j-1,d_{\Xcal}}]$.

Second, $G^{2j}$ computes the $j$-th row and column of $U$ and $L$, respectively. Linear transformations compute $u_{jk}$ and the terms required for computing $l_{kj}$ for $k\geq j$. The linear operation to compute $u_{jk}$ is $u_{jk} = a_{jk} - \sum_{i=1}^{j-1} l_{ji}u_{ik}$. For each $u_{jk}$, the projection matrix $V^{2j}_k$ is a row vector representing the required linear combination of inputs from $A$ and the outputs of layer $G^{2j-1}$, followed by the function $g(x)=x$. Additionally, in the same layer, linear transformations obtain the vector $[n_{kj}=a_{kj} - \sum_{i=1}^{j-1} l_{ki}u_{ij};\ u_{jj}=a_{jj}-\sum l_{ji}u_{ij}]$. The non-linearity $g(x,y) = x/y$ computes elements of $L$, with $l_{kj} = n_{kj}/u_{jj}$.  Proceeding iteratively gives the LU decomposition of $A$.

The next stage solves the system $Ly=b$. At step $i=1, \dots, d_\Xcal$, we have input $[z_0;\ vec(L);\ vec(U);\ \{y_j\}_{j=1}^{i-1}] $ to compute $y_i$. The layer $G^{2d_{\Xcal}+2i-1}$ computes the products $\{l_{ij}y_j\}_{j=1}^{i-1}$. To obtain each term, the projection matrix selects the components $l_{ij}$ and $y_j$ from the input, and the product non-linearity $g(x,y)=xy$ is applied. In the second layer $G^{2d_{\Xcal}+2i}$, linear maps transform the input to $ [b_i - \sum_{j=1}^{i-1} (l_{ij}y_j); l_{ii}]$. A subsequent non-linearity $g(x,y) = x/y$ gives $y_i=(b_i - \sum_{j=1}^{i-1} l_{ij}y_j)/l_{ii} $. This procedure iteratively computes $y \in \mathbb{R}^{d_\Xcal}$. 

Proceeding similarly with backward substitution for the system $Ux=y$, gives the solution $x$.
The complete solver is a composition of these layers. The total depth is hence $6d_\Xcal$. All non-linear operations are two-dimensional, so $d_{\text{max}}=2$. The maximal width is dominated by the LU decomposition stage, giving $\ell_{\text{max}} = \mathcal{O}(d_\Xcal^2)$. 
\end{proof}
\begin{proposition}
     \label{lemma:newt_lc} Suppose Assumption \ref{assump:pic_encoder} holds.
Then, $\Phi_{\text{N}}: C_0(\bar{D}) \to C_0(\bar{D})$, as defined in \eqref{eqn:newton}, exhibits the low complexity structure in Definition \ref{defn:low_comp}, with parameters $k= m(6d_\Xcal+3), d_{\max} = 2$, and $\ell_{\max}= \mathcal{O}(d_{\Xcal}^2)$.
\end{proposition}
\begin{proof}

We show that each iteration of Newton's method, exhibits the defined low-complexity structure. The full $m$-iteration solver is a composition of $m$ such structures. Denote $\mathrm{u}^0 = E_\Xcal(u^0)$ for $u^0 \in C_0(\bar{D})$.  Since, we perform the analysis on the encoded space, we write the discretized version of a single Newton iterate as:
\[
\mathrm{u}^{(1)} =  \mathrm{u}^0 - [I-LM_{\Delta tf'( \mathrm{u}^0)}]^{-1} b, \text{ where } b = \mathrm{u}^0- L (\mathrm{u}^0+ \Delta t f(\mathrm{u}^0))\,,
\]
where $f$ is applied pointwise on $\mathrm{u^0}$. The matrix $L$ corresponds to the numerical integration weights of the linear operator $[1-\Delta t \Delta]^{-1} $, and $M_{\Delta t f'(\mathrm{u}^0)}= \text{diag}\left(\Delta t f'(\mathrm{u}^0)\right)$.

The first layer, $G^1$, takes the input vector $\mathrm{u}^0 \in \mathbb{R}^{d_\Xcal}$. It constructs a vector $z_1 = [\mathrm{u}^0;\ f(\mathrm{u}^0);\ \Delta t f'(\mathrm{u}^0)] \in \mathbb{R}^{3d_\Xcal}$ where $f'$ is applied pointwise. To do this, for $j \in \{1, \dots, d_\Xcal\}$, set $g_j^1(x)=x$ and the projection matrix is taken as the standard basis vector $V_j^1 = e_j \in \mathbb{R}^{d_\Xcal}$. For $j \in \{d_\Xcal+1, \dots, 2d_\Xcal\}$, we set $g_j^1(x)=f(x)$ and $V_j^1 = e_{j-d_\Xcal} \in \mathbb{R}^{d_\Xcal}$. To obtain the last term, for $j \in \{2d_\Xcal+1, \dots, 3d_\Xcal\}$, we take $g_j^1(x)=\Delta tf'(x)$ and $V_j^1 = e_{j-2d_\Xcal} \in \mathbb{R}^{d_\Xcal}$.
 
The second layer, $G^2$, performs linear operations to form the components of the linear system to be solved. It outputs a vector $z_2 = [\mathrm{u}^0;\ b;\ vec(I-L M_{\Delta tf'(\mathrm{u}^0)})] \in \mathbb{R}^{2d_\Xcal +d_\Xcal^2}$. For $j \in \{1, \dots, d_\Xcal\}$, the projection matrices are taken as the standard basis vector, i.e., $V_j^2 = e_j \in \mathbb{R}^{1\times d_\Xcal}$. The vector $b$ can be expressed as a linear map acting on $z_1$: $b=[I-L;\ -\Delta t L;\ \mathbf{0}]z_1$.
Hence, for each $j$ in $\{d_\Xcal+1,\  \dots,\  2d_\Xcal\}$, $V_j^2$ is the $(j-d_\Xcal)$-th row of this block matrix. The associated nonlinearities are simply the function $g_j^2(x)=x$. For $j \in \{ 2d_\Xcal+1, \dots, 3d_\Xcal\}$, projection matrices corresponding to the numerical integration weights of $L$, return $vec(L M_{\Delta tf'(\mathrm{u}^0)})$. A one-dimensional affine non-linearity $g^2_j= \delta_j-x,\ \delta_j \in \{0,1\}$, gives $vec(I-L M_{\Delta tf'(\mathrm{u}^0)})$. This assembles all components of a linear system to be solved. From Lemma \ref{lemma:LU_lc}, we know that solving the linear system can be written in the low complexity structure with maximal width $\mathcal{O}(d_\Xcal^2)$. Concatenating these layers returns the solution $\tilde{\mathrm{u}}$ to the linear system.

 A final linear layer performs the first Newton update $\mathrm{u^{(1)}} \approx \mathrm{u}^0 - \tilde{\mathrm{u}}$. This establishes the low complexity structure for the full $m$-step method, with depth $k= m(6d_\Xcal+3)$, maximal complexity ${d_{\max} = 2}$, and maximal dimension $\ell_{\max}=\mathcal{O}(d_\Xcal^2)$.

\end{proof}

\begin{figure}
    \centering
    \includegraphics[width=1\linewidth]{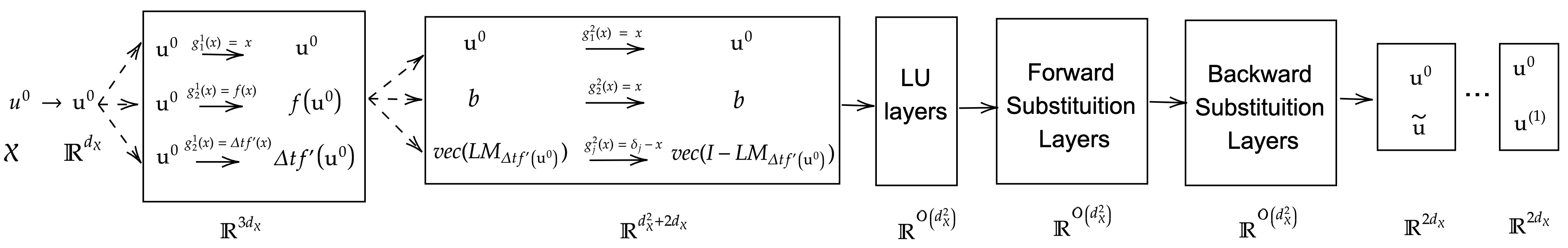}
 \vspace{-6mm}   \caption{Low Complexity Structure for Newtons method. Dashed lines indicate linear transformations, solid lines indicate non-linearities. All transformations are applied pointwise. Note $g^2_j(x):=\delta_j-x$ in the second layer where $\delta_j \in \{0,1\}$ depends on the index. } \vspace{-4mm}
    \label{fig:newt_lc}
\end{figure}

By construction, the output encoder dimension is the same as that of the input, i.e., $d_\Ycal = d_\Xcal$. Our results lead to the following bound on the generalization error.

\begin{theorem}\label{thm:newt_gen}
Suppose Assumptions \ref{assump:compact_supp}, \ref{assump:pic_encoder}, \ref{assump:lipchitz_reaction}-\ref{assump:init_iterate} hold. Let $\GNN$ be the minimizer of the optimization (\ref{eqn:optimization}) where the target operator is $\Phi_{\text{N}}:C_0(\bar{D}) \to C_0(\bar{D})$ defined in \eqref{eqn:newton} for $m$ iterations of Newton's method. Here, the network architecture $\mathcal{F}_\mathrm{NN}(d_{\Xcal},m(6d_\Xcal+3)L,p,M)$ is as in (\ref{eqn:FNN}) with parameters,
\begin{equation*}
\begin{aligned}
		Lp =  \Omega \left( d_{\Xcal}^{\frac{1}{4}} n^{\frac{1}{4}} \right)
  \,, M \geq \sqrt{\ell_{\text{max}}} L_{E_{\Xcal}} R_{\Ycal} \,, \ \ \ell_\text{max} =\mathcal{O}(d_{\Xcal}^2),
\end{aligned}
\end{equation*}	
Let
$L_N = K^m + \frac{2(K^m-1)}{K-1}$ with $K = 3 + 2\Delta t L_f + \mathcal{O}(\Delta t^{2})$. Then, we get,
	\begin{equation*}
		\begin{aligned}
\mathcal{E}_{\text{gen}}\lesssim  L_{{N}}^{2} \log(L_{{N}})(\ell_{\text{max}})^{5/2} n^{-\frac{1}{2}} \log n + \mathcal{E}_{\text{proj}} \,,
		\end{aligned}
	\end{equation*}
	where the constants in $\lesssim$ and $\Omega(\cdot)$ solely depend on $p,\ell_\text{max},R_\Ycal,L_{E_{{\Xcal}}},$ $L_{D_{{\Xcal}}}$, and $\mathcal{O}(\Delta t^2)$ terms depending on $\Omega_0, D,$ and $f$. 
\end{theorem}
\begin{proof}
Propositions \ref{lemma:newt_lipschitz}-\ref{lemma:newt_lc} prove the result as a consequence of Theorem \ref{thm:low_complexity}.
\end{proof}

Theorem \ref{thm:newt_gen} shows that there is no CoD when learning the implicit time-stepping method with a Newton solver. The generalization error depends polynomially on $d_\Xcal$. Unsurprisingly, the error decreases with the step-size much like in classical methods.

\begin{remark}
The time-stepping operator corresponding to the explicit Euler method does not possess Lipschitz continuity, unlike the implicit method. Consider the heat equation, where the explicit update step is:
\(
    u^{1} = [1+ \Delta t \Delta] u^0.
\)
With initial conditions $u^0(x) = \sin(jx),\ j \in \mathbb{N}$, and $v^0(x) = 0$, we find:
\begin{align*}
    \|u^1(x) -v^1(x)\| \leq j^2 \|\sin(jx)\|_\infty = j^2 \|u^0 - v^0\|_\infty.
\end{align*}
Since $j$ can be arbitrarily large, this operator is not Lipschitz continuous, and hence outside the scope of our study.
\end{remark}
Next, we aim to obtain similar results for a parabolic system with forcing terms.

\section{Parabolic Equations with Forcing Terms}
\label{sec:para_gen}

 Consider a general parabolic system with a forcing function of the form:
\begin{align}
\begin{cases}
    &\partial_t u(x,t) = \Delta u(x,t)+ f(x), \quad  x\in D \subseteq \mathbb{R}^d,\ t \in [0,T]  \\
    &u(x,0) = u^0(x), \quad x \in D,
\end{cases}\label{eqn:gen_para}
\end{align}
where $f \in C(\bar{D})$ is the forcing function. For simplicity, we assume homogeneous Dirichlet boundary conditions on $D \subset \mathbb{R}^d$ with a sufficiently smooth boundary $\partial D$. 

\subsection{Implicit Euler Method} The time-stepping operator $\Phi_{\text{B}}:C_0(\bar D) \to C_0(\bar D)$ corresponding to the implicit Euler solver, for fixed forcing $f$, is defined as:
\begin{align}
     u^{1} = [1-\Delta t \Delta ]^{-1}[u^0 + \Delta tf(x)] = \Phi_{\text{B}}(u^0).
     \label{eqn:imp_para}
\end{align}
The corresponding multi-input operator $\Psi_{\text{B}}:C_0(\bar{D})\times C(D)\to C_0(\bar D)$ is defined as:
\begin{equation}\label{eqn:imp_para_multi}
    u^{1} = [1-\Delta t \Delta ]^{-1}[u^0 + \Delta tf(x)] = \Psi_{\text{B}}(u^0,f).
\end{equation}
First, we prove the following result about the single-input operator \eqref{eqn:imp_para}.
\begin{assumption} \label{assump:para_encoder}
    The input $u^0 $ is encoded via discretization encoder to $\mathbb{R}^{d_\Xcal}$.
\end{assumption}

\begin{theorem} \label{thm:be_para}
Suppose Assumptions \ref{assump:compact_supp},\ref{assump:para_encoder} hold. Let $\GNN$ be the minimizer of the optimization (\ref{eqn:optimization}) for target operator $\Phi_{\text{B}}:(C_0(\bar{D}),\|\cdot\|_\infty) \to (C_0(\bar{D}),\|\cdot\|_\infty) $ defined in \eqref{eqn:imp_para}. For network architecture $\mathcal{F}_\mathrm{NN}(d_{\Xcal},L,p,M)$ in (\ref{eqn:FNN}) with parameters:
\begin{equation*}
\begin{aligned}
		Lp =  \Omega \left( d_{\Xcal}^{\frac{1}{2}} n^{\frac{1}{6}} \right)
  \,, M \geq \sqrt{\ell_{d_\Xcal}} L_{E_{\Xcal}} R_{\Ycal} \,,
\end{aligned}
\end{equation*}	
we have the following bound on the generalization error,
	\begin{equation*}
		\begin{aligned}
\mathcal{E}_{\text{gen}}\lesssim  (d_\Xcal)^{3} n^{-\frac{2}{3}} \log n + \mathcal{E}_{\text{proj}} \,.
		\end{aligned}
	\end{equation*}
	The constants in $\lesssim$ and $\Omega(\cdot)$ solely depend on $p,d_\Xcal,R_\Ycal,L_{E_{{\Xcal}}},$ and $L_{D_{{\Xcal}}}$. 
\end{theorem}
\begin{proof} First, we show that $\Phi_{\text{B}}$ is Lipschitz continuous. Let $x_0 \in D$ be the point at which $u^{1}$ achieves a maximum, i.e., $\Delta u^{1}(x_0) \leq 0$. This implies,
\begin{align*}
    u^1(x_0) - \Delta t \Delta u^1(x_0) = u^0(x_0) + &\Delta tf(x_0) \implies u^1(x_0) \leq u^0(x_0)+ \Delta t f(x_0) ,\\
    \implies \max(u^1) &\leq \|u^0\|_\infty + \Delta t \|f\|_\infty.
\end{align*}
We can similarly derive an estimate for $\min(u^1)$. It follows that,
\begin{equation}
    \|u^1\|_\infty \leq \|u^0\|_\infty + \Delta t \|f\|_\infty,.
    \label{eqn:lip_para_be}
\end{equation} 
Note that $u^{1}$ is obtained through an affine transformation of $u^0$, specifically, by applying a linear integral transform of the $u^0$ with an added forcing function. This, together with equation \eqref{eqn:lip_para_be}, implies that the operator is 1-Lipschitz. Furthermore, it has the low complexity structure that is observed by setting $V^1_k \in \mathbb{R}^{d_\Xcal\times 1}$ to be $k$th row of the numerical integration weight matrix $L$ corresponding to the linear operator $[1-\Delta t \Delta]^{-1}$ for $1 \leq k \leq d_\Xcal$. To this, we apply pointwise the non-linearity $g_k(u) = u+ [1-\Delta t \Delta]^{-1} f$. The maximal non-linearity degree $d_{\max} =1$. By construction, the encoding dimension of the output space is same as that of the input, i.e., $d_\Ycal=d_\Xcal$.
This low complexity structure along with the Lipschitz estimate proves the theorem as a consequence of Theorem \ref{thm:low_complexity}. 
\end{proof}

Using intermediate results from the proof of Theorem \ref{thm:be_para}, we obtain a similar result for the multi-input operator.
\begin{assumption} \label{assump:multi_para_encoder}
   The initial condition and forcing function are encoded via discretization encoders to $\mathbb{R}^{d_{\Xcal_1}}$, and $\mathbb{R}^{d_{\Xcal_2}}$, respectively, with $d_\Xcal:= d_{\Xcal_1} + d_{\Xcal_2}$.
\end{assumption}

\begin{theorem} \label{thm:be_para_multi}
Suppose Assumptions \ref{assump:compact_supp},\ref{assump:multi_para_encoder} hold. Let $\GNN$ be the minimizer of the optimization (\ref{eqn:optimization}) for target operator $\Psi_{\text{B}}:C_0(\bar{D})\times C(\bar{D})\to C_0(\bar D)$ defined in \eqref{eqn:imp_para_multi}. Then, for network architecture $\mathcal{F}_\mathrm{NN}(d_{\Xcal},L,p,M)$ in (\ref{eqn:FNN}) with parameters:
\begin{equation*}
\begin{aligned}
		Lp =  \Omega \left( d_{\Xcal}^{\frac{1}{2}} n^{\frac{1}{6}} \right)
  \,, M \geq \sqrt{\ell_{d_\Xcal}} L_{E_{\Xcal}} R_{\Ycal} \,,
\end{aligned}
\end{equation*}	
and $L_{\text{B}}= \max\{\Delta t, 1\}$, we have,
	\begin{equation*}
		\begin{aligned}
\mathcal{E}_{\text{gen}}\lesssim  L_{\text{B}}(d_\Xcal)^{3} n^{-\frac{2}{3}} \log n + \mathcal{E}_{\text{proj}} \,,
		\end{aligned}
	\end{equation*}
	where the constants in $\lesssim$ and $\Omega(\cdot)$ solely depend on $p,d_\Xcal,R_\Ycal,L_{E_{{\Xcal}}},$ and $L_{D_{{\Xcal}}}$. 
\end{theorem}
\begin{proof}
   The proof of Theorem \ref{thm:be_para} shows that $\|u^1\|_\infty \leq \|u^0\|_{\infty} + \Delta t \|f\|_{C^1}$. This boundedness property of the bilinear target operator implies Lipschitz continuity, and the linearity naturally gives it the low complexity structure, thereby Theorem \ref{thm:low_complexity} holds. The maximal width $\ell_{\max} = d_\Xcal$, and the maximal nonlinearity dimension $d_{\max}=1$, as the solution operator is linear.
\end{proof}

\subsection{Crank-Nicholson Method} Next, consider the Crank-Nicholson method for \eqref{eqn:gen_para}. We define the operator $\Phi_{\text{C}}:C^2_0(\bar D) \to C_0(\bar D)$ as:
\begin{align}
    u^{1} =  \left[1-\frac{\Delta t}2 \Delta \right]^{-1}\left[\left(1+\frac{\Delta t}2 \Delta\right) u^0 + \Delta t f(x)\right] = \Phi_{\text{C}}(u^0).
    \label{eqn:cn_para}
\end{align} The following generalization error bound may be obtained for this method.
\begin{theorem} \label{thm:cn_para}
    Suppose Assumptions \ref{assump:compact_supp}, \ref{assump:para_encoder} hold. Let $\GNN$ be the minimizer of optimization (\ref{eqn:optimization}) for target operator $\Phi_{\text{C}}:C_0^2(D) \to C_0(\bar{D})$ defined in \eqref{eqn:cn_para}. Then, for network architecture $\mathcal{F}_\mathrm{NN}(d_{\Xcal},L,p,M)$ in (\ref{eqn:FNN}) with parameters:
\begin{equation*}
\begin{aligned}
		Lp =  \Omega \left( d_{\Xcal}^{\frac{1}{2}} n^{\frac{1}{6}} \right)
  \,, M \geq \sqrt{\ell_{d_\Xcal}} L_{E_{\Ycal}} R_{\Ycal} \,,
\end{aligned}
\end{equation*}	
we have,
	\begin{equation*}
		\begin{aligned}
\mathcal{E}_{\text{gen}}\lesssim  (\ell_{\text{max}})^{3} n^{-\frac{2}{3}} \log n + \mathcal{E}_{\text{proj}} \,.
		\end{aligned}
	\end{equation*}
	The constants in $\lesssim$ and $\Omega(\cdot)$ only depend on $p,\ell_\text{max},R_\Ycal,L_{E_{{\Xcal}}},L_{E_{{\Ycal}}}$ and $L_{D_{{\Xcal}}}, L_{D_{{\Ycal}}}$. 
\end{theorem}
\begin{proof}
    Let $(\lambda_i, \phi_i)_{i=1}^{\infty}$ be the eigenvalue-eigenvector pairs of the negative Laplacian operator $-\Delta$, i.e., $-\Delta \phi_i = \lambda_i \phi_i$. For $u^1 = \Phi_{\text{C}}(u^0)$, we write:
\[
u^0(x) = \sum_{j=1}^\infty U_j^0 \phi_j(x),\ u^{1}(x) = \sum_{j=1}^\infty U_j^{1} \phi_j(x), \ f(x) = \sum_{j=1}^\infty f_j \phi_j(x), \ \text{for } U^0_j, U_j^{1}, f_j \in \mathbb{R},
\]
and replace into the Crank-Nicholson scheme to get,
\[
 \sum_{j=1}^\infty \left( \frac{U_j^{1} - U_j^0}{\Delta t} + \frac{\lambda_j}2 U_j^{1}  + \frac{\lambda_j}2 U_j^{0} - f_j \right) \phi_j(x) = 0.
\]
Using the fact that $\phi_j$ are orthogonal to each other, after some rearrangement:
\begin{equation} \label{eqn:cn_int}
    U_j^{1} = \frac{1-\frac{\Delta t}{2} \lambda_j }{1+\frac{\Delta t}{2} \lambda_j}
 U_j^0+ \frac{\Delta t}{1+\frac{\Delta t}{2} \lambda_j}f_j \quad \forall j \geq 1.
\end{equation}
Let $v^1= \Phi_{\text{C}}(v^0)$ for $v^0 \in \Omega_\Xcal$. We write $v^0(x) = \sum_{j=1}^\infty V_j^0 \phi_j(x),$ and $v^{1}(x) = \sum_{j=1}^\infty V_j^{1} \phi_j(x)$. For Dirichlet boundary conditions, $\lambda_j>0 \  \forall \ j$, implying $\left|\frac{1-\frac{\Delta t}{2} \lambda_j }{1+\frac{\Delta t}{2} \lambda_j}\right| \leq 1$. Thus,
\begin{align*}
    \left|V^1_j - U^1_j\right| \leq  \left|V^0_j - U^0_j\right| 
    \implies \|v^1 - u^1\|_\infty \leq \|v^0 -u^0\|_\infty \leq \|v^0-u^0\|_{C^2},
\end{align*}
demonstrating that the operator is Lipschitz. It also clearly has the low complexity structure as it consists of a linear integral transformation followed by a 1-dimensional non-linearity (as in Theorem \ref{thm:be_para}). Applying Theorem \ref{thm:low_complexity} proves the result.
\end{proof}

Under Assumption \ref{assump:multi_para_encoder}, similar results as in Theorem \ref{thm:be_para_multi} holds for the multi-input Crank-Nicholson operator $\Psi_{\text{C}}:C_0^2(\bar D) \times C(\bar{D}) \to C_0(\bar{D})$ defined as: $$\Psi_{\text{C}}(u^0,f) :=\left[1-\frac{\Delta t}2 \Delta \right]^{-1}\left[\left(1+\frac{\Delta t}2 \Delta\right) u^0 + \Delta t f(x)\right].$$
As the eigenvalues $\{\lambda_i\}$ of the negative Laplacian are positive, from \eqref{eqn:cn_int}, it follows that $\|\Psi_{\text{C}}(u^0,f) \|_\infty \leq \|u^0\|_{C^2}+\Delta t \|f\|_{\infty}$. So the Lipschitz constant for the operator is the same as that of $\Psi_{\text{B}}$. As this operator is again linear, the low complexity structure has maximal nonlinearity $d_{\max}=1$, and maximal width $\ell_{\max}=d_{\Xcal}$. Hence, we get the same generalization error estimate with the same order network parameters.

In the next section, we establish similar results for viscous conservation laws.

\section{Conservation Laws} \label{sec:cons_gen}

Conservation equations are a class of PDEs that describe the evolution of quantities conserved over time. They have applications in fluid dynamics \cite{perthame2002kinetic}, traffic flow \cite{colombo2003hyperbolic}, and gas dynamics \cite{rozhdestvenski_1983systems}, see \cite{leveque1992numerical,leveque2002finite,sod1978survey} and the references therein for numerical methods for these equations. However, discontinuities may arise in finite time for the solution, even for smooth initial data. Analytical solutions are only available in simple cases - such as linear flux functions, simple initial data, or when the method of characteristics can be effectively applied.  

Introducing a viscosity coefficient $\kappa>0$ regularizes the problem by smoothing out discontinuities that arise from the hyperbolic terms \cite{Hopf1950}. The second-order term introduces diffusion, which dampens high-frequency oscillations and ensures the existence of well-behaved, physically meaningful solutions. Moreover, the viscosity term can be viewed as a modeling tool that captures physical processes such as heat conduction or internal friction. 
Viscous conservation laws take the general form:
\begin{align}
\begin{cases}
     &{\partial_t u (x,t)} + \partial_x{f}(u(x,t)) = \kappa \Delta u(x,t),\ x \in D \subseteq \mathbb{R}^d,\ t \in (0,T)\,,
    \\
    &u(x,0) = u^0(x) ,\ x \in D,
\end{cases}\label{eqn:cons_law}
\end{align}
where \( u \) represents the conserved quantity (such as density or momentum), and \( {f}(u) \in C^1(\mathbb{R}) \) is the flux function that governs the flow of \( u \) through space. We assume homogeneous Dirichlet boundary conditions and a sufficiently smooth boundary $\partial D$. Burgers' equation is a classic example of a conservation law with $f(u) = \frac{1}{2}u^2$.

For $u^0 \in H_0^1(D)$, discretizing \eqref{eqn:cons_law} only in time using the Backward Euler method gives the semi-discrete solution $u^1$ at time $\Delta t$:
\begin{align}
    u^1 &= u^0 - \Delta t \partial_x f(u^1) + \kappa \Delta t \Delta u^1 \nonumber \\
    \iff u^1 &= \left[1 - \kappa \Delta t \Delta\right]^{-1}\left(u^0  - \Delta t \partial_x f(u^1)\right) =: \Phi_{\text{B}}(u^0, u^1).
    \label{eqn:imp_cons2}
\end{align}
Equation \eqref{eqn:imp_cons2} may be solved for $u^1$ using Picard's fixed point method with iterations: 
\begin{align} 
    u^{(0)} = u^0,\  u^{(i)}  =\Phi_{\text{B}}(u^0, u^{(i-1)}) \text{ for } 0 < i\leq m,\ u^1 \approx u^{(m)} =: \Phi_{\text{P}}(u^0). \label{eqn:cons_pic}
\end{align}
The operator $\Phi_{\text{P}}:H^1_0(\bar D) \ni u^0 \mapsto u^{(m)} \in H^1_0(\bar D)$ refers to the BE method used in conjunction with $m$-steps of Picard's iteration. We assume that the iterates themselves are equipped with the homogeneous Dirichlet boundary conditions imposed on the continuous formulation \eqref{eqn:cons_law}. Note that when $m=1$, the implicit method solved with Picard's fixed point algorithm is equivalent to the IMEX scheme.

Similarly, the multi-input BE operator $\Psi_{\text{P}}:H_0^1(\bar D)\times C^1(\mathbb{R})\to H_0^1(\bar D)$ for the conservation law is defined as:
\begin{align} \label{eqn:cons_pic_multi}
 u^{(i)}=  \left[1 - \kappa \Delta t \Delta\right]^{-1}\left(u^0  - \Delta t \partial_x f(u^{(i-1)})\right),0 < i\leq m,\   u^{(m)} =: \Psi_{\text{P}}(u^0,f),
\end{align}
 We derive generalization error bounds for both cases in this section.
For simplicity, we take the viscosity coefficient $\kappa$ to be above a certain threshold. First we consider the case where the input is solely the initial condition.
\begin{assumption} \label{assump:lip_flux}
    The flux function is $L_f$-Lipschitz continuous.
\end{assumption}

\begin{proposition} \label{prop:cons_lip}
   Let $\Omega_0$  be a compact subset of $H^1_0(D)$. Suppose Assumption \ref{assump:lip_flux} holds, and $\kappa \geq \frac{ {L}_f}2$. The time-stepping operator $\Phi_{\text{P}}: \Omega_0  \to  H_0^1(D) $ defined in \eqref{eqn:cons_pic} is Lipschitz continuous. 
   For $
   L_{\text{P}}:= (1+\Delta t L_f)/{\min\bigl\{1,\Delta t \left(2\kappa - {{L}_f}\right)\bigr\}}$, we have:
   \[\|\Phi_{\text{P}}(u^0) - \Phi_{\text{P}}(v^0)\|_{H^1} \leq L_{\text{P}}^{\frac m2}\|u^0 - v^0\|_{H^1}, \quad \text{ for } u^0,v^0 \in \Omega_0.\] 
\end{proposition}
\textit{Proof:} Denote $w^0 = u^0 -v^0$ and $w^{(i)} = u^{(i)} - v^{(i)}$ for $u^{(i)}= \Phi_{\text{B}}(u^0, u^{(i-1)})$ and $v^{(i)}= \Phi_{\text{B}}(v^0, v^{(i-1)})$ for $1 \leq i \leq m$. Then, by definition: 
\begin{align}
    w^{(m)} 
    &= w^{(m-1)} - \Delta t \partial_x \left(f(u^{(m-1)})-f(v^{(m-1)})\right) + \kappa \Delta t \Delta w^{(m)}.
\end{align}
Multiplying by $w^{(m)}$ and integrating over $D$, after some rearrangement, we obtain:
\begin{align}
    \int_{D} \left[\left(w^{(m)}\right)^2 - w^{(m)} w^{(m-1)} \right]\,dx + \Delta t &\int_{D} w^{(m)} \partial_x \left(f(u^{(m-1)})-f(v^{(m-1)})\right) \,dx \nonumber \\
    &= \Delta t \kappa \int_{D} w^{(m)} \Delta w^{(m)} \,dx.
\end{align}
The inequality $\frac{b^2-a^2}{2} \leq  b^2-ab $ gives a lower bound for the first term:
\begin{align}
     \frac12 \int_{D} \left[\left(w^{(m)}\right)^2 - \left(w^{(m-1)}\right)^2 \right]\,dx + \Delta t \int_{D} &w^{(m)} \partial_x \left(f(u^{(m-1)})-f(v^{(m-1)})\right) \,dx \nonumber \\ \label{eqn:cons_step}
     &\leq  \Delta t \kappa \int_{D} w^{(m)} \Delta w^{(m)} \,dx.
\end{align}
Next, integrating by parts the last two terms with the boundary terms vanishing due to homogeneous Dirichlet boundary conditions:
\begin{align}
     \frac12 \int_{D} \left[(w^{(m)})^2 - (w^{(m-1)})^2 \right]\,dx &- \Delta t \int_{D} \partial_x w^{(m)} \left(f(u^{(m-1)})-f(v^{(m-1)})\right) \,dx \nonumber \\
     &\leq  -\Delta t \kappa \int_{D}  (\partial_x w^{(m)})^2 \,dx. \label{eqn:cons_ineq_pic}
\end{align}
Using the Lipschitz property of the flux function, and then Young's inequality,
\begin{align*}
    \int_{D} \partial_x w^{(m)} \left(f(u^{(m-1)})-f(v^{(m-1)})\right)& \,dx \leq L_f \int_{D} |\partial_x w^{(m)}| \cdot |w^{(m-1)}| \,dx \\
    &\leq \frac{L_f}2 \left[\int_{D} \left(\partial _x w^{(m)}\right)^2 \,dx +  \int_{D} \left(w^{(m-1)}\right)^2 \,dx\right]
\end{align*}
Substituting this estimate into \eqref{eqn:cons_ineq_pic}, and rearranging:
\begin{align}
        \int_{D} (w^{(m)})^2 \,dx +\Delta t \left(2\kappa - {{L}_f}\right) \int_{D}  (\partial_x w^{(m)})^2 \,dx &\leq   ({1+\Delta t L_f}) \int_{D} (w^{(m-1)})^2 \,dx \nonumber\\
        &\leq (1+\Delta t L_f)\|w^{(m-1)}\|_{H^1}^2\label{eqn:pic_cons_bd}
\end{align}
When $\kappa \geq \frac{{L}_f}2$, the second term on the left side of the inequality is positive. Thus, when the diffusion term dominates, applying \eqref{eqn:pic_cons_bd} recursively:
\begin{align}
  \min \bigl\{1,\Delta t \left(2\kappa - L_f\right) \bigr\}  \|w^{(m)}\|^2_{H^1} \leq  \left(1+\Delta t L_f\right)^m \|w^0\|^2_{H^1}.
\end{align}
This gives the desired result.

\begin{assumption}  \label{assump:cons_encoder}
   A discretization encoder is used to encode $u^0\in H^1_0(\bar{D})$ to $\mathbb{R}^{d_\Xcal}$.
\end{assumption}

\begin{proposition} \label{prop:cons_lc}
  Suppose Assumption \ref{assump:cons_encoder} holds. The operator $\Phi_{\text{P}}: H_0^1(D) \to H_0^1(D)$ defined in \eqref{eqn:cons_pic}, under Assumption \ref{assump:cons_encoder},  has the low complexity structure in Definition \ref{defn:low_comp} with $d_{\max} =2$, and $l_{\max}=3d_{\Xcal}$.
\end{proposition}
\begin{proof}
    We show that each iteration of the mapping maintains a low-complexity structure. Let $T$ denote the numerical differentiation matrix corresponding to the partial derivative operator $\partial_x$, and $L$ denotes the numerical integration weight matrices corresponding to $[1-\Delta t\Delta]^{-1}$. Then, in the first layer, for $j \in \{1, \dots, d_\Xcal\}$, we set both rows of $V^1_j\in \mathbb{R}^{d_\Xcal\times 2}$ to be the Dirac-delta vector at $x_j$ and $g^1_j(x,y)=x$. For $j \in \{d_\Xcal +1, \dots, 2d_\Xcal\}$, let $k= j - d_\Xcal$. We set $V^1_j$ so that its first row is the Dirac-delta vector at $x_k$ and the second row is the $k$th row of the numerical differentiation weights $T$. The corresponding non-linearity is $g^1_j(x,y) = f'(x)y$. Finally, for $j \in \{2d_\Xcal +1,\dots, 3d_\Xcal\}$, let $\tilde{k}= j -2 d_\Xcal$. We define $V^1_j$ similarly with the first row being the Dirac-delta vector at $x_{\tilde{k}}$ and the second row is the ${\tilde{k}}$-row of the numerical integration weights $L$, with nonlinearity $g^1_j(x,y) = y.$ 
    To summarize, after the linear transformations, the encoded input $\mathrm{u}^0=E_\Xcal(u^0)$ is mapped to: \[\begin{bmatrix}
        \mathrm{u}^0,\mathrm{u}^0;\, \mathrm{u}^0,T \mathrm{u}^0;\, \mathrm{u}^0,L\mathrm{u}^0 \,
    \end{bmatrix} \in \mathbb{R}^{3d_\Xcal \times 2}.\]   
    Then applying the two-dimensional non-linearities specified above, we get as output $[\mathrm{u}^0;\ f'(\mathrm{u}^0) T \mathrm{u}^0;\ L\mathrm{u}^0] \in \mathbb{R}^{3d_\Xcal}$. A linear transformation in the next layer gives: \[[\mathrm{u}^0;\   L\mathrm{u}^0 -\Delta t  L\left(f'(\mathrm{u}^0\right) T \mathrm{u}^0 ) ]^T \in \mathbb{R}^{2d_\Xcal}.\] Notably, the second term corresponds to the first iterate. Extending this process iteratively establishes the desired result.
\end{proof}

\begin{theorem} \label{thm:imp_cons}
   Suppose Assumptions \ref{assump:compact_supp}, \ref{assump:lip_flux}-\ref{assump:cons_encoder} hold, and $\kappa \geq L_f/2$. Let $\GNN$ be the minimizer of optimization (\ref{eqn:optimization}) where the target operator $\Phi_{\text{P}}: L^2(D) \to L^2(D)$ is defined in \eqref{eqn:cons_pic}. Then for the network architecture $\mathcal{F}_\mathrm{NN}(d_{\Xcal},2mL,p,M)$ defined in (\ref{eqn:FNN}) with parameters:
\begin{equation*}
\begin{aligned}
		Lp =  \Omega \left( d_{\Xcal}^{\frac{1}{4}} n^{\frac{1}{4}} \right)
  \,, M \geq \sqrt{\ell_{d_\Ycal}} L_{E_{\Ycal}} R_{\Ycal} \,,
\end{aligned}
\end{equation*}	
 we have for $
   L_{\text{P}}:= (1+\Delta t L_f)/{\min\bigl\{1,\Delta t \left(2\kappa - {{L}_f}\right)\bigr\}}$,
	\begin{equation*}
		\begin{aligned}
\mathcal{E}_{\text{gen}}\lesssim  L_{\text{P}}^m(d_\Xcal)^{\frac{5}{2}} n^{-\frac{1}{2}} \log n + \mathcal{E}_{\text{proj}} \,,
		\end{aligned}
	\end{equation*}
	where the constants in $\lesssim$ and $\Omega(\cdot)$ solely depend on $m, p,\ell_\text{max},R_\Ycal,L_{E_{{\Xcal}}},$ and $L_{D_{{\Xcal}}}$. 
\end{theorem}
\begin{proof}
    Propositions \ref{prop:cons_lip}-\ref{prop:cons_lc} imply that the operator satisfies the assumptions of Theorem \ref{thm:low_complexity}. The result follows as a consequence.
\end{proof}

\begin{remark}
    For $m=1$, Theorem \ref{thm:imp_cons} establishes a generalization error estimate for first order IMEX methods for conservation laws.
\end{remark}

Finally, we establish the generalization error estimate for the multi-input operator using intermediate results from the proof of Theorem \ref{prop:cons_lip}.

\begin{assumption}
    \label{assump:flux_enco} A basis encoder $\eqref{eqn:basis}$ encodes the flux $f\in \Omega_p \subset C^1(\mathbb{R})$ to $\mathbb{R}^{d_{\Xcal_2}}$.
\end{assumption}

\begin{theorem} 
   Suppose Assumptions \ref{assump:compact_supp}, \ref{assump:lip_flux}-\ref{assump:flux_enco} hold, and $\kappa \geq (1+L_f)/2$. Let $\GNN$ be the minimizer of optimization (\ref{eqn:optimization}) where the target operator $\Psi_{\text{P}}: H_0^1(D) \times C^1(\mathbb{R}) \to H_0^1(D)$ is defined in \eqref{eqn:cons_pic_multi}. For the network architecture $\mathcal{F}_\mathrm{NN}(d_{\Xcal},2mL,p,M)$ defined in (\ref{eqn:FNN}) with parameters:
\begin{equation*}
\begin{aligned}
		Lp =  \Omega \left( d_{\Xcal}^{\frac{1}{4}} n^{\frac{1}{4}} \right)
  \,, M \geq \sqrt{\ell_{\max}} L_{E_{\Ycal}} R_{\Ycal} \,, \ \ell_{\max}= d_{\Xcal_1}d_{\Xcal_2}+d_{\Xcal_1}+d_{\Xcal_2},
\end{aligned}
\end{equation*}	
 we have for $
   L_{\text{M,P}}:= (1+\Delta t L_f)/{\min\bigl\{1,\Delta t \left(2\kappa - {{L}_f}\right)\bigr\}}$, 
	\begin{equation*}
		\begin{aligned}
\mathcal{E}_{\text{gen}}\lesssim   L_{\text{M,P}}^m(d_\Xcal)^{\frac{5}{2}} n^{-\frac{1}{2}} \log n + \mathcal{E}_{\text{proj}} \,,
		\end{aligned}
	\end{equation*}
	where the constants in $\lesssim$ and $\Omega(\cdot)$ solely depend on $m, p,D, \ell_\text{max},R_\Ycal,L_{E_{{\Xcal}}},$ and $L_{D_{{\Xcal}}}$. 
\end{theorem}
The proof follows Theorem \ref{thm:imp_cons} with modifications for the multi-input case. A Lipschitz estimate on the operator may be obtained by modifying \eqref{eqn:cons_step} and using triangle inequality to decompose the corresponding flux terms and bounding each term separately using similar techniques as in the proof. The low complexity structure follows as in Proposition \ref{thm:gen_err_pic}, by noting that the highest degree of non-linearity comes from learning the two-dimensional non-linearities of individual components of the basis expansion $g_i(a,x) = a T_i(x)$ in \eqref{eqn:basis}. The maximal width comes from storing the intermediate representations leading to $\ell_{\max}=d_{\Xcal_1}d_{\Xcal_2}+d_{\Xcal_1}+d_{\Xcal_2}$. This gives us the bound, as a consequence of Theorem \ref{thm:low_complexity}.

\section{Global Error Accumulation}\label{sec:global_error}
We now analyze the error propagation when a neural network, $\GNN$, is used to recursively predict the solution over a long time horizon $T=N\Delta t$. Let $u(x,t)$ be the solution to \eqref{eqn:gen_pde} for $u^0 \in \Omega_\Xcal$. Let $\GNN$ be the minimizer of the optimization problem \eqref{eqn:optimization} trained on data $S=\{u_i^0,u_i^{1}\}$. The data pairs are generated using the time-stepping algorithm $\Phi_{\Delta t}$, i.e., $u_i^1 = \Phi_{\Delta t}(u_i^0)$ where $u_i^1$ is the approximate solution at $\Delta t$. The global accumulation error is defined as the expected deviation from the true solution $u(x,T)$ after $N$ steps:

\begin{align}
   \EgenN = \mathbb{E}_{u^0\sim \gamma} \left[\| u(T) - \underbrace{\Phi_{\text{NN}} \circ \dots \circ \Phi_{\text{NN}}}_{N \text{ times }} (u^0) \|_{\Ycal} \right] \label{eqn:egenn}
\end{align}
To study this error, we make the following assumptions on the analytical solution.
\begin{assumption} \label{assump:glob_bd}
    For initial condition $u^0 \in \Omega_\Xcal$, $u(x,t) \in C^2([0,T],\Ycal)$. Also, $\sup_{u^0 \in \Omega_\Xcal} \sup_{0 \leq t \leq T} \|u'(t)\|_{\Ycal} \leq L_0$ and $\sup_{u^0 \in \Omega_\Xcal} \sup_{0 \leq t \leq T} \|u''(t)\|_{\Ycal} \leq M_0$.
\end{assumption}
\begin{proposition}
   Suppose Assumptions \ref{assump:compact_supp}-\ref{assump:low_complexity} and \ref{assump:glob_bd} hold. Let $\GNN$ be the minimizer of the optimization (\ref{eqn:optimization}) where the training data $\mathcal{S} = \{(u_i^0, \Phi_{\Delta t}(u^0_i))\}_{i=1}^n$ for the target first-order time-stepping scheme $\Phi_{\Delta t}: \Xcal \to \Ycal$. Then, for $T=N\Delta t$:
\begin{align}
    \EgenN \leq \frac{M_0\Delta t (e^{T L_0}-1)}{2L_0} + \sum_{i=1}^{N}  L^{i-1}_{\Phi_{\Delta t}}\mathcal{E}_{gen} 
\end{align}
\end{proposition}

\begin{proof}
Let $\Phi_{\text{NN}}^N := \underbrace{\Phi_{\text{NN}} \circ \dots \circ \Phi_{\text{NN}}}_{N \text{ times }} (u^0)$ and $\Phi_{\Delta t}^N:= \underbrace{\Phi_{\Delta t} \circ \dots \circ \Phi_{\Delta t}}_{N \text{ times }}$. First, using the triangle inequality, \eqref{eqn:egenn} can be split into:
\begin{align}
\mathbb{E}_{u^0\sim \gamma} \left[\| u(T) - \Phi_{\text{NN}}^N (u^0) \|_{\Ycal} \right] &\leq \mathbb{E}_{u^0\sim \gamma} \left[ \| u(T) - \Phi_{\Delta t}^N (u^0)\|_{\Ycal} \right] \nonumber\\
&+\mathbb{E}_{u^0\sim \gamma} \left[\| \Phi_{\Delta t}^N(u^0) - \Phi_{\text{NN}}^N (u^0) \|_{\Ycal} \right] \label{eqn:glob_1}
\end{align}
The first term in \eqref{eqn:glob_1} is bounded  by the global error of first order methods  \cite{Atkinson1989}:
\begin{align}
  \mathbb{E}_{u^0\sim \gamma} \left[ \| u(T) - \Phi_{\Delta t}^N (u^0)\|_\Ycal \right]  \leq  \frac{M_0\Delta t \left(e^{TL_0}-1\right)}{2L_0}.
    \end{align}
The second term in \eqref{eqn:glob_1} can be decomposed into:
\begin{align}
   \mathbb{E}_{u^0\sim \gamma} &\left[\| \Phi_{\Delta t}^N (u^0) - \Phi_{\text{NN}}^N  (u^0) \|_{\Ycal} \right]  \leq \mathbb{E}_{u^0\sim \gamma} \left[\| \Phi_{\Delta t}^{N} (u^0) - \Phi_{\Delta t}^{N-1} \circ \Phi_{\text{NN}} (u^0)  \|_{\Ycal} \right] \nonumber\\
   &+ \mathbb{E}_{u^0\sim \gamma} \left[\|  \Phi_{\Delta t}^{N-1} \circ \Phi_{\text{NN}} (u^0)  - \Phi_{\text{NN}}^N  (u^0) \|_{\Ycal} \right].\label{eqn:glob_2}
\end{align}
Using the Lipschitz property of the time-stepping operator:
\begin{align*}
     \mathbb{E}_{u^0\sim \gamma} \left[\| \Phi_{\Delta t}^{N} (u^0) - \Phi_{\Delta t}^{N-1} \circ \Phi_{\text{NN}} (u^0)  \|_{\Ycal} \right] &\leq  L^{N-1}_{\Phi_{\Delta t}} \Egen,\ \text{(from \eqref{eqn:gen_err})}.
\end{align*}
Decomposing the second term in \eqref{eqn:glob_2} recursively in a similar fashion and applying the Lipschitz estimate will give the bound on the total error:
\begin{align*}
    \mathbb{E}_{u^0\sim \gamma} \left[\| u(T) - \Phi_{\text{NN}} ^N(u^0) \|_{\Ycal} \right] \leq  \frac{M_0\Delta t (e^{N L_0\Delta t}-1)}{2L_0} + \sum_{i=1}^{N}  L^{i-1}_{\Phi_{\Delta t}}\mathcal{E}_{gen}.
\end{align*}
The global error accumulation thus is a combination of the error accrued by the numerical solver being approximated, and the network's learning error.

\end{proof}

\section{Future Work and Conclusion}\label{sec:concl}

In this work, we have demonstrated that semi-discrete time-stepping schemes for a wide-class of PDEs exhibit an inherent low-complexity structure and Lipschitz continuity. These properties allow neural networks to effectively learn these schemes, overcoming the curse of dimensionality.

The inviscid hyperbolic conservation law, which we do not study in this paper, poses challenges due to the lack of smoothing effects from diffusion, making it difficult to establish Lipschitz continuity for the semi-discrete method. Classical tools such as fixed-point methods and variational principles are often inapplicable, and techniques effective for scalar equations, like comparison principles, do not easily extend to systems \cite{bressan2000hyperbolic}. Entropy solutions for scalar conservation laws do satisfy an 
$L^1$-contracting semigroup property \cite{krushkov1970first}, that is preserved by certain fully discrete numerical schemes \cite{kuznetsov1976accuracy, cockburn1994error, sanders1983convergence, vila1994convergence}. However, whether an operator learning model with an appropriate choice of encoder-decoders maintains this stability property remains an open question.

In Neural ODEs \cite{chen2018neural}, the derivative of the hidden state is parameterized by a neural network, and a numerical solver, such as the Euler method, is used for integration. However, unlike the approach presented in this work, the solver is not directly learned, which may limit its ability to bypass the computational costs of traditional methods. Extending the methods proposed in this study to incorporate the framework of Neural ODEs could provide a promising direction for future research.

As further work, we would like to extend our analysis to meta-learning systems \cite{ li2016learning} and other learned iterative schemes. Building on our theory that DNNs can learn Newton's method, it would be interesting to study how networks learn other optimization algorithms like gradient descent \cite{andrychowicz2016learning} and to provide quantitative estimates on their generalization error by analyzing the structure and stability of these  operators.

\bibliographystyle{abbrv}
\bibliography{ref.bib}

\end{document}